\documentclass[11pt,a4]{amsart}
\usepackage{amsfonts,amsmath,amssymb,bbm}
\usepackage{verbatim}

\makeatletter
\renewcommand{\section}{\@startsection{section}{1}{0pt}{20pt}{6pt}{\large\bfseries}}
\makeatother

\usepackage{enumerate}
\usepackage{pdfsync}

\numberwithin{equation}{section}

\theoremstyle{plain}
  \newtheorem{theorem}{Theorem}[section]
  \newtheorem{proposition}[theorem]{Proposition}
  \newtheorem{lemma}[theorem]{Lemma}
   \newtheorem{corollary}[theorem]{Corollary}
   
\theoremstyle{remark}

  \newtheorem{remark}[theorem]{Remark}

\def \\ { \cr }
\def\R{\mathbb{R}}
\def \1{1 \mkern -6mu 1}
\def\N{\mathbb{N}}
\def\E{\mathbb{E}}
\def\P{\mathbb{P}}

\def \e{{\bf e}}

\newcommand{\LL}{L\'{e}vy }
\newcommand{\LLP}{L\'{e}vy process }
\newcommand{\LLPs}{L\'{e}vy processes }

\newcommand{\Id}[1]{\mathbb{I}_{\{#1\}}}
\newcommand{\Tb}{\mathcal{T}_{\beta}}
\newcommand{\Ph}{\mathcal{P}}
\newcommand{\rb}{\right)}
\newcommand{\lb}{\left(}
\newcommand{\Ec}[1]{\mathbb{E}\left[#1\right]}

\begin{document}
\title[Exponential Functional of L\'evy processes]{Extended factorizations of exponential functionals of L\'evy processes}

\author{P. Patie}

\address{Department of Mathematics, Universit\'e Libre de Bruxelles, B-1050 Bruxelles, Belgium.}
\email{ppatie@ulb.ac.be}

\author{M. Savov}
\address{New College, University of Oxford, Oxford, Holywell street OX1 3BN, UK.}
 \email{savov@stats.ox.ac.uk, mladensavov@hotmail.com}

%\thanks{The two last authors are grateful to A.E. Kyprianou for stimulating organic discussions while they were visiting the University of Bath.}

\begin{abstract}
In \cite{Pardo-Patie-Savov}, under mild conditions, a Wiener-Hopf type factorization is derived for the exponential functional of  proper L\'evy processes. In this paper, we extend this  factorization by relaxing a finite moment assumption as well as by considering the exponential functional  for killed L\'evy processes. As a by-product, we derive some interesting fine distributional properties enjoyed by a large class of this random variable, such as the absolute continuity of its distribution and the smoothness, boundedness or complete monotonicity of its density. This type of results is then used to derive similar properties for the law of maxima and first passage time of some stable L\'evy processes. Thus, for example, we show that for any stable process with $\rho\in(0,\frac{1}{\alpha}-1]$, where $\rho\in[0,1]$ is the positivity parameter and $\alpha$ is the stable index, then the first passage time has a bounded and non-increasing density on $\R_{+}$. We also generate many instances of integral or power series representations for the law of the exponential functional of L\'evy processes with one or two-sided jumps.  The proof of our main results requires different devices from the one developed in \cite{Pardo-Patie-Savov}. It relies in particular on a generalization  of a transform recently introduced in \cite{Chazal-al-10} together with some extensions to killed L\'evy process of  Wiener-Hopf techniques.
\end{abstract}



\maketitle

\section{Introduction and main results}\label{Sec1}
Let $\xi=(\xi_t)_{t\geq0}$ be a possibly killed L\'evy process starting from $0$. We denote by $\Psi_q$ its L\'evy-Khintchine exponent  which takes the form, for any $z\in i\R$,
\begin{equation}\label{Levy-K}
\Psi_q(z)=bz +\frac{\sigma^{2}}{2}z^{2}+\int_{-\infty}^{\infty}\left(e^{zy} - 1 - zy \mathbb{I}_{\{|y|<1\}}\right)\Pi(dy)-q,
\end{equation}
where $q\geq 0$ is the killing rate, $\sigma\geq0, b \in \R$ and $\Pi$ is a sigma-finite positive measure satisfying the condition
$\int_{\mathbb{R}}(y^{2}\wedge 1)\Pi(dy)<\infty$.  In this paper, we are interested in both characterizing the distribution and deriving some fine distributional properties of the so-called  exponential functional of $\xi$, which is defined by
\begin{equation*}%\label{eq:expfunc}
{\rm{I}}_{\Psi_q}=\int_0^{\infty}e^{\xi_t}\mathbb{I}_{\{t<\e_q\}}dt,
\end{equation*}
where  $\e_q$ is the lifetime of $\xi$, i.e.~it is an exponential random variable of parameter $q$ (with the convention that $\e_0=\infty$) independent of $\xi$.  When $q=0$ we simply write $\Psi=\Psi_0$ and we assume that $\xi$ drifts to $-\infty$.  
The motivation for studying this positive random variable finds its roots in probability theory but has some strong connections with
issues coming from other fields of mathematics such as functional and complex analysis. Besides their inherent interest,
problems of this type have also ties with other areas of sciences, e.g. astrophysics, biology, insurance and
mathematical finance. It is also worth mentioning that there exists a close connection between the law of the exponential functional of some specific L\'evy processes and the one of the maxima of stable processes offering a way to study the fluctuation of these processes from a perspective different from the classical Wiener-Hopf techniques. 
We refer to \cite{Pardo-Patie-Savov} for a thorough description of the recent methodologies which have been developed to investigate the distribution of ${\rm{I}}_{\Psi_q}$. In particular, we mention that, in that  paper, it is shown under a mild assumption that,  when $q=0$ and $-\infty<\E\left[\xi_{1}\right]<0$,  the variable ${\rm{I}}_{\Psi}$ factorizes into the product of two independent exponential functionals of L\'evy processes defined in terms of the ladder height processes of $\xi$. The purpose of this paper is to extend this Wiener-Hopf type factorization by first relaxing   the finite moment condition on the underlying L\'evy processes and then by deriving similar factorization identities for the exponential functional of killed L\'evy processes. We emphasize that the approach carried out in \cite{Pardo-Patie-Savov} can not be used to deal with this generalization. Indeed, therein, the main identity is obtained by means of the functional equation \eqref{eq:rec_sub1}, satisfied by the Mellin transform of ${\rm{I}}_{\Psi}$ combined with the characterization of its distribution as the stationary measure of some generalized Ornstein-Uhlenbeck processes. Indeed the law of ${\rm{I}}_{\Psi_q}$, for any $q>0$, cannot be identified as a stationary measure of some Markov process anymore whereas  when $q=0$ and the first moment of $\xi_1$ is not finite, the functional equation \eqref{eq:rec_sub1} does not hold even on the imaginary line and therefore we cannot directly guess the existence of a probabilistic factorization. In order to circumvent these difficulties our strategy relies on a transformation between Laplace exponents of L\'evy processes which allows to establish a connection between the study of the exponential functional for unkilled and killed L\'evy processes. This will be achieved by generalizing to our context a mapping recently introduced by Chazal et al.~\cite{Chazal-al-10} and by providing some interesting  results concerning the Wiener-Hopf factorization of killed L\'evy processes. We also indicate that our extended factorizations of exponential functionals allow  us to identify some fine distributional properties enjoyed by a large class of these random variables, such as the smoothness of their distribution, the monotonicity,  complete monotonicity of their density, etc. We will be using these type of results to provide some new distributional properties enjoyed by the density of first passage times for some stable L\'evy processes.\newline 

\noindent In order to state  our main result we introduce some notation. First, we recall that the reflected processes
 $\left(\sup_{0\leq s\leq t}\xi_s-\xi_t\right)_{t\geq 0}$ and $\left(\xi_t-\inf_{0\leq s\leq t}\xi_s\right)_{t\geq 0}$ are Feller processes in $[0,\infty)$ which possess local times $(L^{\pm}_t)_{t\geq0}$ at level $0$, see \cite[Chapter IV]{Bertoin-96}. The ascending and descending ladder times are defined as the right-continuous inverse of $L^{\pm}$, viz. $(L^{\pm}_t)^{-1}=\inf\{s> 0;\:  L^{\pm}_s>t\}$
 and the ladder height processes
 $H^+$ and $H^-$ by
 $$H^+_t=\xi_{(L^{+}_t)^{-1}}=\sup_{0\leq s\leq (L^{+}_t)^{-1}}\xi_s\,, \qquad \hbox{ whenever }(L^{+}_t)^{-1}<\infty\,,$$
  $$H^-_t=\xi_{(L^{-}_t)^{-1}}=\inf_{0\leq s\leq (L^{-}_t)^{-1}}\xi_s\,, \qquad \hbox{ whenever } (L^{-}_t)^{-1}<\infty\,.$$
Here, we use the convention that $\inf\varnothing =\infty$ and $H^{\pm}_t=\infty$, when $L^{\pm}_{\infty}\leq t$. 
%and $H^{-}(t)=-\infty$ when $L^{-}_{\infty}\leq t$.
From \cite[p. 27]{Doney}, we have for $q\geq 0$, $s\geq 0$,
 \begin{equation}\label{BivLadder}
\log{ \E\left[e^{-q (L^{+}_1)^{-1}-s H^{+}_1}\right]} = -\Phi_+(q,s)=-k_{+}-\eta_{+}q-\delta_{+}s-\int_{0}^{\infty}\int_{0}^{\infty}\Big(1-e^{-(q y_{1}+s y_{2})}\Big)\mu_{+}(dy_{1},dy_{2}),
 \end{equation}
 where $\eta_{+}$ (resp.~$\delta_{+}$) is the drift of the subordinator $(L^{+})^{-1}$ (resp.~$H^+$) and $\mu_{+}(dy_{1},dy_{2})$ is the \LL measure of the bivariate subordinator $\lb (L^{+})^{-1},H^{+}\rb $. Similarly, for $q\geq 0$, $s\geq 0$,
 \begin{equation}\label{BivLadder1}
\log  \E\left[e^{-q (L_1^{-})^{-1}+s H^{-}_1}\right]=-\Phi_-(q,s)=-\eta_{-}q-\delta_{-}s-\int_{0}^{\infty}\int_{0}^{\infty}\Big(1-e^{-(q y_{1}+s y_{2})}\Big)\mu_{-}(dy_{1},dy_{2}),
 \end{equation}
 where  $\eta_{-}$ (resp.~$\delta_{-}$) is the drift of the subordinator $(L^{-})^{-1}$ (resp.~$-H^{-}$ ) and $\mu_{-}(dy_1,dy_2)$ is the \LL measure of the bivariate subordinator $\left((L^{-})^{-1},-H^{-}\right)$.
The celebrated Wiener-Hopf factorization then reads off as
\begin{equation} \label{eq:wh}
\Psi_q(z) = - \Phi_+(q,-z)\Phi_-(q,z),
\end{equation}
where we set $\Phi_+(1,0)=\Phi_-(1,0)=1$ as the normalization of the local times. We point out that while it can happen that $((L^{+})^{-1},H^{+})$  $\lb \text{resp.~} \lb (L^{-})^{-1},-H^{-})\rb\rb$ can be increasing renewal processes, see \cite[Section 1]{Bingham}, this does not affect our definitions.
% We simply write  $ \phi_-(z)=- \Phi_-(0,z)$ and $\phi_+(z)=-\Phi_+(0,-z)$.
As in \cite{Pardo-Patie-Savov} throughout the paper we work with the following set of measures:
\[ \mathcal{ P:
} \quad \textrm{\em the set of positive measures on } \R_{+} \textrm{ \em which admit a non-increasing density.}\]

Our first theorem is the main result in our paper. Equation \eqref{MainAssertion} extends \cite[(1.6), Theorem 1.2]{Pardo-Patie-Savov} to the killed case as well to the case when $\E\left[\xi_{1}\right]=-\infty$ and it is the backbone of all our applications. 
\begin{theorem}\label{MainTheorem}
Let $q\geq0$ and assume that $\xi$ drifts to $-\infty$, when $q=0.$ Then the law of the random variable ${\rm{I}}_{\Psi_{q}}$ is absolutely continuous with density which we denote by $m_{\Psi_{q}}$. Next, assume that one of the following conditions holds:
\begin{enumerate}
\item
\begin{description}
\item [P$+$] $\Pi_+ (dy)=\Pi(dy) \mathbb{I}_{\{y>0\}} \in \mathcal{ P}$,
\end{description}
\item
\begin{description}
\item[P${}^q_{\pm}$]  $\mu_{q_+}(dy)= \int_{0}^{\infty}e^{-q y_{1}}\mu_{+}(dy_{1},dy) \in \mathcal{ P}$, $\mu_{q_-}(dy)= \int_{0}^{\infty}e^{-q y_{1}}\mu_{-}(dy_{1},dy) \in \mathcal{ P}$.
\end{description}
\end{enumerate}
Then, in both cases, there exists an unkilled spectrally positive L\'evy process with a negative mean such that its Laplace exponent $\psi^{q_+}$ takes the
form
\begin{equation}\label{Phi}\psi^{q_+}(-s)=s\Phi_+(q,s)=\delta_{+}s^{2}+ q_+ s+s^{2}\int_{0}^{\infty}e^{-sy}\mu_{q_+}(y,\infty)dy,\: s\geq0,\end{equation}
where  $q_+=k_+ +\eta_{+}q+\int_{0}^{\infty}\int_{0}^{\infty}\left(1-e^{-q y_{1}}\right)\mu_{+}(dy_{1},dy_{2})>0$.
Furthermore, for any $q\geq0$, we have the factorization
\begin{equation}\label{MainAssertion}
{\rm{I}}_{\Psi_q}\stackrel{d}={\rm{I}}_{\phi_{q_-}}\times {\rm{I}}_{\psi^{q_+}},
\end{equation}
where $\times$ stands for the product of independent random variables and  $\phi_{q_-}(z)=-\Phi_-(q,z)$ is the Laplace exponent of a negative of a subordinator which is killed at the rate given by the expression \newline $q_-=k_- +\eta_{-}q+\int_{0}^{\infty}\int_{0}^{\infty}\left(1-e^{-q y_{1}}\right)\mu_{-}(dy_{1},dy_{2})\geq 0$ and $q_-=0$ if and only if $q=0$.
\end{theorem}
\begin{remark}
 \begin{enumerate}
 \item   We mention that when $q=0$, in comparison to \cite[Theorem 1.1]{Pardo-Patie-Savov}, here we also include the case when $\E\left[\xi_{1}\right]=-\infty$. We recall that under such a condition, the functional equation  \eqref{eq:rec_sub1} below does not even hold on the imaginary line $i\R$.
% \item   In comparison again to \cite[Theorem 1]{Pardo-Patie-Savov}, we do not consider here the case ${\bf E_+}$. However, an extension of the result in this framework is provided in Lemma \ref{lem:me} below under slightly more restrictive conditions.
 \item We emphasize that the main factorization identity  \eqref{MainAssertion} allows to build up many examples of two-sided L\'evy processes  for which the density of ${\rm{I}}_{\Psi_q}$ can be described as a convergent power series. This is due to the fact that the exponential functionals on the right-hand side of the identity are easier to study as we have, for instance, simple  expressions for their positive or negative integer moments. More precisely, the positive entire moments of ${\rm{I}}_{\phi_{q_-}}$, for any $q\geq0$, are given in \eqref{eq:ms} below and we have from \cite{Bertoin-Yor-02} that the law of  $1/{\rm{I}}_{\psi^{q_+}}$ is determined by its positive entire moments as follows
\begin{eqnarray} \label{eq:msn}
\E[{\rm{I}}_{\psi^{q_+}}^{-m}] &=&-(\psi^{q_+})'(0^-)\frac{\prod_{k=1}^{m-1}\psi^{q_{+}}(-k)}{\Gamma(m)}, \: m=1,2,\ldots
\end{eqnarray}
with the convention that the right-hand side is $-(\psi^{q_+})'(0^-)$ when $m=1$.
Some specific examples will be detailed  in Section \ref{sec:proof-cor}.
 \item Assuming that we start with bivariate Laplace exponents $\Phi_+$ and
 $\Phi_-$ such that their L\'evy measures satisfy condition {\bf {P}}${}^q_{\pm}$ with $\Phi_+(q,0)\Phi_-(q,0)>0$, then from Lemma \ref{lem:pk} below we can construct a killed L\'evy process with Laplace exponent $\Psi_q$ given by identity \eqref{eq:wh} and such that factorization  \eqref{MainAssertion}  holds.

\item We point out that \eqref{MainAssertion} holds even when $\phi_{q_-}(z)-\phi_{q_-}(0)=0$, for all $z\in i\R$, i.e. when $\xi$ is a subordinator. In this case  $I_{\phi_{q_-}}=\int_{0}^{\e_{q_-}}ds=\e_{q_-}$.
 \end{enumerate}
\end{remark}
We postpone the proof of the theorem to  Section \ref{sec:pm}. We proceed instead by providing some consequences of our factorization identity \eqref{MainAssertion}   concerning some interesting distributional properties of the exponential functional. Before stating the results,  we recall that the density $m$ of  a positive random variable is completely monotone if $m$ is infinitely continuously differentiable and $(-1)^nm^{(n)}(x)\geq0$, for all $x\geq0$ and $n=0,1,\ldots$ . Note in particular, that $m$ is non-increasing and thus the distribution of the random variable is unimodal with mode at $0$, that is its distribution is concave on $[0,\infty)$.
\begin{corollary}\label{cor:main}
\begin{enumerate}[(i)]
\item Let us assume that either condition (1) or (2) of Theorem \ref{MainTheorem} and $|\Psi_{q}(s)|<+\infty$, for $s\in[-1,0]$, holds true. Then, for any $q>0$, such that $\Psi_{q}(-1)\leq0$, the density $m_{\Psi_q}$ is non-increasing, continuous and a.e. differentiable on $\R^+$ with $m_{\Psi_q}(0)=q$.
\item Let $\xi$ be a subordinator with L\'evy measure  $\Pi_+ \in \mathcal{P}$.    Then, for any $q>0$, the density of ${\rm{I}}_{\Psi_q}$ is completely monotone and bounded with $m_{\Psi_q}(0)=q$. Moreover,  recalling that, in this case, the drift  $b$ of $\xi$ is non-negative, we have, for any $x<1/b$ (with the convention that $1/0=+\infty$), that
\[m_{\Psi_q}(x)=\sum_{n=0}^{\infty} a_n\left(\Psi_{q}\right)\frac{(-x)^n}{n!},\]
where $a_n\left(\Psi_{q}\right) = q\prod_{k=1}^{n}-\Psi_{q}(-k)$ with $a_0\left(\Psi_{q}\right) = q$. If $b>0$, we  have for any $x>0$
\begin{equation} \label{eq:pse}
m_{\Psi_q}(x)=(1+bx)^{-1}\sum_{n=0}^{\infty} \tilde{a}_n\left(\Psi_{q}\right)\left(\frac{bx}{bx+1}\right)^n,
\end{equation}
where $\tilde{a}_n\left(\Psi_{q}\right)= \sum_{k=0}^n\frac{a_k\left(\Psi_{q}\right)}{k!(n-k)!}$.

\item Let $\xi$ be a spectrally positive L\'evy process and we denote,  for any $q>0$, by  $\gamma_q$  the only positive root of the equation $\Psi_q(-s)=0$. Then, we have
\begin{equation}\label{eq:hy}
{\rm{I}}_{\Psi_q}\stackrel{d}= B^{-1}(1,\gamma_q
) \times {\rm{I}}_{\psi^{q_+}},
\end{equation}
where  $B(1,\gamma_q)$ is a Beta random variable and $\psi^{q_+} (z)=z\phi_{q_+}(z) $. Moreover, if $\Psi_q(-1)\leq0$ then ${\rm{I}}_{\Psi_q}$ has a non-increasing density.
\item Finally, let   $\xi$ be a spectrally negative L\'evy process. Denoting here, for any $q>0$, by  $\gamma_q$  the only positive root of the equation $\Psi_q(s)=0$,  we have
\begin{equation}\label{SpecNeg1}
m_{\Psi_{q}}(x)=\frac{x^{-\gamma_q-1}}{\Gamma(\gamma_q)}\int_{0}^{\infty}e^{-y/x}y^{\gamma_q}m_{\phi_{q_-}}(y)dy, \: x>0,
\end{equation}
where $\Gamma$ stands for the Gamma function. In particular, we get the precise asymptotic for the density at infinity, i.e.
\[\lim_{x\rightarrow \infty }x^{\gamma_q+1}m_{\Psi_q}(x) = \frac{\E\left[{\rm{I}}_{\phi_{q_-}}^{\gamma_q}\right]}{\Gamma(\gamma_q)}. \]
Then, for any $\beta\geq\gamma_q+1$, the mapping $x\mapsto x^{-\beta}m_{\Psi_q}(x^{-1})$ is completely monotone. In particular, the density of the random variable ${\rm{I}}^{-1}_{\Psi_q}$ is completely monotone, whenever $\gamma_q \leq 1$.
\end{enumerate}
\end{corollary}

\begin{remark} \label{rem:cor_main}
\begin{enumerate}
\item We point out that a positive random variable with a completely monotone density is in particular infinitely divisible, see \cite[Theorem 51.6]{Sato-99}.
\item A positive  random variable with a non-increasing  density is strongly multiplicative unimodal (for short SMU), that is the product of this random variable with any independent positive random variable is unimodal and in this case the product has its  mode at $0$, see \cite[Proposition 3.6]{Cuculescu}.
\item We mention that the two cases (ii) and (iii), i.e. when the L\'evy process has no negative jumps, have not been studied in the literature so far.
\item One may recover from item (iv) the expression of the density found in \cite{Patie-abs-08} in this case.
Furthermore, we point out that in \cite{Patie-abs-08} it is proved that the density extends actually to a function  of a complex variable which is analytical  in the entire complex plane cut along the negative real axis and admits a power series representation for all $x>0$.
\item Note that the series \eqref{eq:pse} is easily amenable to numerical computations since $a_k\lb\Psi_{q}\rb$ can be computed recurrently. We stress that \eqref{eq:pse} would be difficult to get from \eqref{eq:rec_sub1} because the functional equation holds on a strip in the complex plane which might explain why such simple series has not yet appeared in the literature. 
\end{enumerate}
\end{remark}
The proof of this corollary and of the following one will be given in Section \ref{sec:proof-cor}. We will also describe therein some examples illustrating these results. As a by-product of Corollary \ref{cor:main}, we get the following new distributional properties for the law of maximum and first passage times of some stable L\'evy processes.
\begin{corollary}\label{cor:stable}
Let $Z=(Z_t)_{t\geq0}$ be an $\alpha$-stable L\'evy process starting from $0$ with $\alpha \in (0,2]$. Let us write $S_1 = \sup_{0\leq t\leq 1 } Z_t $ and $T_1=\inf\{t>0; \:  Z_t \geq 1\}$ and recall that the scaling property of $Z$ yields the identity $T_1 \stackrel{d}{=}S_1^{-\alpha}$. Then, writing $\rho=\P(Z_1>0)$, we have the following claims:
\begin{enumerate}[(i)]
\item The density of $T_1$ is bounded and non-increasing for any $\alpha \in (0,1)$ and $\rho \in \left(0,\frac{1}{\alpha}-1\right]$. In particular, this property holds true for any $\alpha \in (0, \frac{1}{2}] $ and for symmetric stable L\'evy processes, i.e.~$\rho=\frac{1}{2}$, with $\alpha \in (0, \frac{2}{3}]$.
   % \item the density of $S_1$ is non-increasing for any $0<\alpha<1$, $\rho \in (0,\frac{1}{\alpha}-1]$ and  $\rho\neq 1$.
 \item The density of $S_1^{\alpha}$ is completely monotone if $\alpha \in (1,2]$  and  $ \rho = 1-\frac{1}{\alpha}$, that is when $Z$ is spectrally positive.
\end{enumerate}
\end{corollary}
\begin{remark}
When $\alpha \in (0,1)$ and $\rho=1$, that is $Z$ is a stable subordinator, we have the obvious identity $Z_1\stackrel{d}{=}S_1$. Thus, we get  from the first item that the density of $Z_1^{-\alpha}$ is non-increasing on $\R^+$ if $\alpha\in (0,\frac{1}{2}]$. From Remark \ref{rem:cor_main} (2) this means that for these values of $\alpha$,  $Z_1^{-\alpha}$  is SMU with mode at $0$. Note that this result is consistent with the main result of Simon in \cite{Simon} where it is shown that $Z_1$ is SMU with a positive mode if and only if $\alpha\in (0,\frac{1}{2}]$. The positivity of the mode implies that any non-zero real power of $Z_1$, and in particular $T_{1}\stackrel{d}=Z^{-\alpha}_{1}$, is SMU, see e.g.~\cite[Section 1]{Simon}.
\end{remark}

\section{Proof of Theorem \ref{MainTheorem}} \label{sec:pm}

\subsection{The case ${\bf{P}}+$}\label{Sec2}

We start by extending to two-sided L\'evy processes a transformation which has been introduced in \cite{Chazal-al-10} in the framework of spectrally negative L\'evy processes. This mapping turns out to be very useful in the context of both  the Wiener-Hopf factorization of  L\'evy processes and their exponential functionals.  For its statement we need the following notation
  \[\overline{\Pi}_-(y) =  \int_{-\infty}^y\Pi(dr) \mathbb{I}_{\{y<0\}} \textrm{ and }\:\overline{\Pi}_+(y) =  \int^{\infty}_y\Pi(dr) \mathbb{I}_{\{y>0\}} \]
and the following assumption:
 \[{\bf{T_{\beta^+}}}: \: \textrm{There exists } \beta^{+}>0 \textrm{ such that for all } \beta \in [0,\beta^{+}), \: |\Psi(\beta)|<+\infty \textrm{ and   } e^{\beta y}\overline{\Pi}_+(y)dy \in \mathcal{ P}.\]
Also if $\Psi$ satisfies ${\bf{T_{\beta^+}}}$, we write, for all $q\geq 0$,
 \begin{equation}\label{eq:db}
 \beta_q^{*}=\sup\{\beta>0;\: \Psi(\beta)-q<0\}\wedge\beta_{+}.\end{equation}
 \begin{proposition} \label{prop:mapping}
Let  us assume that ${\bf {T_{\beta^+}}}$ holds. Then, for any $\beta \in  (0,\beta_+)$, the linear mapping $\mathcal{T}_{\beta}$ defined by
\begin{eqnarray*}
\mathcal{T}_{\beta}\Psi_{q}(s) &=& \frac{s}{s+\beta}\Psi_{q}(s+\beta), \: s \in (-\beta,\beta_+-\beta),
\end{eqnarray*}
 is the Laplace exponent of an unkilled L\'evy process $\xi^{(\beta,q)}=(\xi^{(\beta,q)}_t)_{t\geq0}$ with Gaussian coefficient $\frac{\sigma^2}{2}$ and  L\'evy measure  $ \Pi^{\beta}$ given by
 \begin{equation} \label{eq:ct_t}
 \Pi^{\beta}(dy)=e^{\beta y}\left(\Pi(dy)-\beta\overline{\Pi}_+(y)dy +\lb\beta\overline{\Pi}_-(y)+q\beta\rb dy\mathbb{I}_{\{y<0\}}\right).
\end{equation}
Furthermore, if we assume that $\xi$ drifts to $-\infty$ when $q=0$, then for any $q\geq0$, $\beta_q^{*}>0$, and, for any $\beta \in (0,\beta_q^{*})$, we have
\begin{equation}\label{Mean1}
-\infty<\Ec{\xi^{(\beta,q)}_{1}}=\frac{\Psi(\beta)-q}{\beta}<0.
\end{equation}
\end{proposition}
\begin{proof}
First, by linearity of the mapping $\mathcal{T}_{\beta}$, one gets
\begin{equation}\label{eq:cp1}
\mathcal{T}_{\beta}\Psi_q (z) = \mathcal{T}_{\beta}\Psi(z)- q\frac{z}{z+\beta},
\end{equation}
where we recognize, on the right-hand side, the  Laplace exponent of a negative of a compound Poisson process with parameter $q>0$, whose jumps are exponentially distributed with parameter $\beta>0$. Next we observe that  one can write
\begin{equation} \label{eq:dp}
\Psi(z)=\Psi_-(z)+\Psi_+(z),
 \end{equation}
where $\Psi_+(z)=\int_{0}^{\infty}\left(e^{zy} - 1 - zy\Id{|y|<1}\right)\Pi_+(dy)$ and $\Psi_-$ is the Laplace exponent of a   L\'evy process without positive jumps. The description of $\mathcal{T}_{\beta}\Psi_-$ as the Laplace exponent of a L\'evy process without positive jumps follows from  \cite[Proposition 2.2]{Chazal-al-10}. Thus, from the linearity of the transform it remains to study its effect on $\Psi_+$.
An integration by parts gives us that
\begin{eqnarray*}
\Psi_+(z) &=& z\left(\int^{\infty}_{0}\left(e^{zy}  -\Id{|y|<1}\right)\overline{\Pi}_+(y)dy +\overline{\Pi}_+(1)\right).
\end{eqnarray*}
Then
\begin{eqnarray*}
\mathcal{T}_{\beta}\Psi_+(z) &=& \frac{z}{z+\beta}\Psi_+(z+\beta)= z\left(\int^{\infty}_{0}\left(e^{(z+\beta)y} - \Id{|y|<1}\right)\overline{\Pi}_+(y)dy+\overline{\Pi}_+(1)\right)\\
&=& z\left(\int^{\infty}_{0}\left(e^{z y} - \Id{|y|<1}\right)e^{\beta y}\overline{\Pi}_+(y)dy+\int^{1}_{0}\left(e^{\beta y} - 1\right)\overline{\Pi}_+(y)dy+\overline{\Pi}_+(1)\right)\\
&=&\int^{\infty}_{0}\left(e^{z y} -1-zy\Id{|y|<1}\right)\left(-e^{\beta y}\overline{\Pi}_+(y)\right)'dy \\ &+&z\left(\int^{1}_{0}\left(e^{\beta y} - 1\right)\overline{\Pi}_+(y)dy+\overline{\Pi}_+(1)(1-e^{\beta})\right),
\end{eqnarray*}
which provides the expression \eqref{eq:ct_t} since the mapping $y \mapsto e^{\beta y}\overline{\Pi}(y)$ is non-increasing by assumption and plainly the  condition ${\bf {T_{\beta^+}}}$  gives that $\int_0^{\infty}(1\wedge y^2)(-e^{\beta y}\overline{\Pi}_+(y))'dy<+\infty$. Next, when $q>0$ then $\beta_q^{*}>0$ since $\Psi(0)=0$ and the mapping $s\mapsto\Psi(s)$ is continuous on $[0,\beta^+)$. When $q=0$,  the condition ${\bf {T_{\beta^+}}}$ combined with the fact that $\xi$ drifts to $-\infty$ implies that $\Psi'(0+)<0$, where $\Psi'(0+)$ can be $-\infty$. Clearly then  we have that $\Psi(\beta)<0$,  for any $\beta \in (0,\epsilon)$ and some $\epsilon>0$, and hence $\beta_0^{*}>0$. Moreover, we observe that, for any $q\geq0$,
\[(\mathcal{T}_{\beta}\Psi_q)'(0^+) = \frac{\Psi(\beta)-q}{\beta},\]
which is clearly finite and negative, for any $\beta \in (0,\beta_q^{*})$.  The proof of the proposition is completed.
\end{proof}
\begin{remark}
We note that, for any $q>0$ and  any $0<\beta<\beta^+$,   the L\'evy process $\xi^{(\beta,q)}$ can be decomposed as $\xi^{(\beta,q)}_t=\xi^{\beta}_t-N^{q}_t$, where $(\xi^{\beta}_t)_{t\geq0}$ is a L\'evy process with Laplace exponent $\mathcal{T}_{\beta}\Psi$ and $(N^{q}_t)_{t\geq0}$ is  an independent compound Poisson process with parameter $q$ whose jumps are exponentially distributed with parameter $\beta$.
\end{remark}
We shall need  the following alternative representation of the bivariate ladder exponents as well as an interesting application of the transform $\Tb$ in the context of  the Wiener-Hopf factorization of L\'evy processes.
\begin{proposition}\label{prop:twh}
For any $q>0$, we have  $\phi_{q_+}(z)=-\Phi_{+}(q,-z)$ and $\phi_{q_-}(z)=-\Phi_{-}(q,z)$, where $\phi_{q_+}$ (resp.~$\phi_{q_-}$) is the Laplace exponent of  (resp.~the negative of ) a subordinator. More precisely, they take the form
 \begin{eqnarray*}\label{BivLadder1}
\phi_{q_+}(z) &=&-q_+ +\delta_{+}z+\int_{0}^{\infty}\left(e^{z y}-1\right)\mu_{q_+}(dy),\\
\phi_{q_-}(z) &=&-q_- -\delta_{-}z-\int_{0}^{\infty}\left(1-e^{-z y}\right) \mu_{q_-}(dy),
 \end{eqnarray*}
 where we recall that $\mu_{q_\pm}(dy)=\int_{0}^{\infty}e^{-qy_1}\mu_{\pm}(dy_{1},dy)$ and $q_{\pm}=k_{\pm}+\eta_{\pm}q+\int_{0}^{\infty}\int_{0}^{\infty}\left(1-e^{-q y_{1}}\right)\mu_{\pm}(dy_{1},dy_{2})>0$.
Consequently, the Wiener-Hopf factorization \eqref{eq:wh} takes the form
\begin{equation} \label{eq:wh1}
\Psi_q(z) = - \phi_{q_+}(z)\phi_{q_-}(z).
\end{equation}
Moreover, assume that ${\bf{T_{\beta^+}}}$ holds and that $\xi$ drifts to $-\infty$ when $q=0$.  Then, for any $\beta\in (0,\beta_q^{*})$, we have
\begin{equation} \label{eq:cp}
\mathcal{T}_{\beta}\Psi_q (z) = -\phi_{q_+}(z+\beta)  \mathcal{T}_{\beta}\phi_{q_-}(z).
\end{equation}
\end{proposition}

\begin{proof}
Since, for any $q>0$, we have that
 \begin{eqnarray*}
 \int_{0}^{\infty}\int_{0}^{\infty}\left(1-e^{-(q y_{1}+z y_{2})}\right)\mu_{\pm}(dy_{1},dy_{2})&=&   \int_{0}^{\infty}\left(1-e^{-z y_{2}}\right) \mu_{q_\pm}(dy_2) \\ &+& \int_{0}^{\infty}\int_{0}^{\infty}\left(1-e^{-q y_{1}}\right)\mu_{\pm}(dy_{1},dy_{2}),
 \end{eqnarray*}
we deduce the first claim from the fact that $q_{\pm}>0$, for any $q>0$.

\noindent Next, we have, under the ${\bf{T_{\beta^+}}}$ condition that $s\mapsto \phi_{q_+}(s)$ is well-defined on $(-\infty,\beta^+)$, see \cite[Lemma 4.2]{Pardo-Patie-Savov}. Also, for any $\beta \in (0,\beta_q^{*})$,  $\phi_{q_+}(\beta)<0$ as clearly both $\Psi_q(\beta)<0$ and $\phi_{q-}(\beta)<0$. Thus, for such $\beta$ the mapping $s\mapsto \phi_{q_+}(s+\beta)$ is the Laplace exponent of a killed subordinator. Moreover, it is not difficult to check that, for any fixed $q\geq0$ and $\beta \in (0,\beta_q^{*})$, $z\mapsto \mathcal{T}_{\beta}\phi_{q_-}(z) $ is  the Laplace exponent of the negative of a proper subordinator. Moreover, since $\beta \in (0,\beta_q^{*})$ we deduce, from the item (2) of Proposition \ref{prop:mapping}, that the proper L\'evy process with characteristic exponent $\mathcal{T}_{\beta}\Psi_q$  drifts to $-\infty$ and hence its  descending ladder height process  is also the negative of a proper subordinator, see e.g.~\cite{Doney}. We conclude by observing the identities
\begin{equation*}%\label{eq:cp}
\mathcal{T}_{\beta}\Psi_q (z+\beta)=\frac{z}{z+\beta}\Psi_q (z+\beta) = -\phi_{q_+}(z+\beta) \frac{z}{z+\beta}\phi_{q_-}(z+\beta)=-\phi_{q_+}(z+\beta)\mathcal{T}_{\beta}\phi_{q_-}(z)
\end{equation*}
and by invoking the uniqueness for the Wiener-Hopf factors, see \cite[Theorem 45.2 (i)]{Sato-99}.
\end{proof}

The $\mathcal{T}_{\beta}$ transform turns out to be also very useful in  proving the following claim which shows, in particular, that the family of exponential functional of L\'evy processes is invariant under some  length-biased transforms. In particular,  the law of  ${\rm{I}}_{\Psi_{q}}$   admits such a  representation in terms of a perpetual exponential functional. Although, a similar result was given  in \cite{Chazal-al-10} for one-sided L\'evy processes, its extension requires deeper arguments.
%In the cases considered in \cite{Chazal-al-10}, one has access to the expression of integer moments of the exponential functional and the uniqueness of  the solution to the functional equation follows readily.
\begin{theorem}\label{Th1}
Let us assume that $\xi$ drifts to $-\infty$ when $q=0.$ Then the following claims hold:
\begin{enumerate}
\item The law of the random variable ${\rm{I}}_{\Psi_{q}}$ is absolutely continuous.
    \item Assume further that ${\bf{T_{\beta^+}}}$ holds. Then,  for every $\beta \in (0,\beta_q^{*})$, there exists a proper L\'evy process  with a finite negative mean and Laplace exponent $\mathcal{T}_{\beta}\Psi_{q}$, such that
\begin{equation}\label{Th1-1}
m_{\Psi_{q}}(x)=\E\left[{\rm{I}}^{\beta}_{\Psi_{q}}\right]x^{-\beta}m_{\mathcal{T}_{\beta}\Psi_{q}}(x)\: \text{a.e., for $x>0$},
\end{equation}
where $m_{\mathcal{T}_{\beta}\Psi_{q}}$ is the density of ${\rm{I}}_{\mathcal{T}_{\beta}\Psi_{q}}$.
\item Finally,  for any $q\geq0$, we have
\begin{equation*}
\lim_{\beta \to 0} {\rm{I}}_{\Tb\Psi_q}\stackrel{d}={\rm{I}}_{\Psi_q}.
\end{equation*}
\end{enumerate}
\end{theorem}
\begin{proof}
First, we point out that the absolute continuity of ${\rm{I}}_{\Psi_q}$ in the case $q=0$ is well-known and can be found in \cite[Theorem 3.9]{Bertoin-Lindner-07}. Thus, we assume that $q>0$.  The case when $\xi$ is with infinitely many jumps can be recovered from \cite[Theorem 3.9 (b)]{Bertoin-Lindner-07}. Indeed, for any $v>0$, denote by $g(s)=e^{s}$ and $dY^{(v)}_{t}=\mathbb{I}_{\{t<v\}}dt$. Since $g(s)$ is strictly increasing we have that condition (3.12) in \cite{Bertoin-Lindner-07} is satisfied. Moreover, for $\epsilon<v$, we have that the density of the absolutely continuous part of the measure $dY^{(v)}_{t}$ restricted to $[0,\epsilon]$ has a density which equals $1$. According to  \cite[Theorem 3.9 (b)]{Bertoin-Lindner-07} this suffices to show that
$\int_{0}^{v}e^{\xi_{s}}ds$ has a law which is absolutely continuous with respect to the Lebesgue measure. Then for any Borel set $A\subset (0,\infty)$ we have that
\[\P\left({\rm{I}}_{\Psi_{q}}\in\,A\right)=q\int_{0}^{\infty}e^{-qt}\P\left(\int_{0}^{t}e^{\xi_{s}}ds\in\,A\right)dt\]
and our statement follows in this case. Next, assume that $\xi=\xi^{(1)}+B$ where $\xi^{(1)}$ is a compound Poisson process and $B$ a Brownian motion with given mean and variance, which can be both zero. We denote by $(T_{n})_{n\geq1}$ (resp.~$(Y_{n})_{n\geq1}$) the sequence of inter-arrival times (resp.~the sequence) of the jumps of $\xi^{(1)}$. Define the measures $\Upsilon$ and $\tilde{\Upsilon}$ respectively on $\R^{\N_{+}}_{+}=\{\omega=(t_{1},t_{2},...):\: t_{i}>0,\text{ for $i\geq 1$} \}$ and $\R^{\N_{+}}=\{\tilde{\omega}=(y_{1},...): y_{i}\in\R,\text{ for $i\geq 1$} \}$ to be induced by the sequences $(T_{n})_{n\geq1}$ and $(Y_{n})_{n\geq1}$. For any $\omega $ and $\tilde{\omega}$, we set $S_{0}(\omega) =\tilde{S}_{0}(\tilde{\omega})=0$ and we write  $S_{j}(\omega)=\sum_{i=1}^{j}t_{i}$, $\tilde{S}_{j}(\tilde{\omega})=\sum_{i=1}^{j}y_{i}$,
and \[P_{j}(\omega)=\P(S_{j}(\omega)\leq \e_{q}<S_{j+1}(\omega))=\P(A_{j}(\omega)).\] Denote by
\[\Gamma_{j,\omega}(dx)=\P\lb \e_{q}\in dx; A_{j}(\omega)|\omega\rb=P_{j}(\omega) \P\lb \e_q\in dx| A_{j}(\omega)\rb\]
and note that $\Gamma_{j,\omega}$ are absolutely continuous with respect to the Lebesgue measure. We also set
\[{\rm{I}}_{k}(\omega)=e^{\tilde{S}_{k-1}(\tilde{\omega})}\int_{S_{k-1}(\omega)}^{S_{k}(\omega)}e^{B_{s}}ds.\]
Now, we pick $A\subset \R_{+}$ such that the Lebesgue measure of $A$ is zero and write
\begin{eqnarray*}
\P\lb {\rm{I}}_{\Psi_{q}}\in A\rb&=& \int_{\R^{\N_{+}}_{+}\times\R^{\N_{+}}}\sum_{j=1}^{\infty}\P\lb \sum_{k=1}^{j}{\rm{I}}_{k}(\omega)+e^{\tilde{S}_{k}(\tilde{\omega})}\int_{S_{k}(\omega)}^{\e_q}e^{B_{s}}ds \in A;\,A_{j}(\omega) \big|\,\omega,\tilde{\omega} \rb\Upsilon(d\omega)\tilde{\Upsilon}(d\tilde{\omega})\\ &=&
\int_{\R^{\N_{+}}_{+}\times\R^{\N_{+}}}\sum_{j=1}^{\infty}\P\lb  \int_{S_{k}(\omega)}^{\e_q}e^{B_{s}}ds\in e^{-\tilde{S}_{k}(\tilde{\omega})}\lb A-\sum_{k=1}^{j}{\rm{I}}_{k}(\omega)\rb;\,A_{j}(\omega) \big|\,\omega,\tilde{\omega}  \rb\Upsilon(d\omega)\tilde{\Upsilon}(d\tilde{\omega}).
\end{eqnarray*}
Next, denote by $\mathbb{D}_{k}=\mathbb{D}_{S_{k}(\omega)}$ the full set of continuous functions images of the Brownian motion up to time $S_{k}(\omega)$ and note that
\begin{eqnarray*}
& &\P\lb  \int_{S_{k}(\omega)}^{\e_q}e^{B_{s}}ds\in e^{-\tilde{S}_{k}(\tilde{\omega})}\lb A-\sum_{k=1}^{j}{\rm{I}}_{k}(\omega)\rb;\,A_{j}(\omega) \big|\,\omega,\tilde{\omega}  \rb\\ &=&
\int_{f\in\mathbb{D}_{k}}\P\lb\int_{0}^{\e_q-S_{k}(\omega)}e^{B'_{s}}ds\in \tilde{A}_{j}(\omega);\,A_{j}(\omega)\big|\,\omega,\tilde{\omega},(B_{s})_{s\leq S_{k}(\omega)}=(f_{s})_{s\leq S_{k}(\omega)}\rb\Theta(df),
\end{eqnarray*}
where $B'$ is an independent copy of $B$, $\tilde{A}_{j}(\omega)=\left\{e^{-\tilde{S}_{k}(\tilde{\omega})-f_{S_{k}(\omega)}}\lb A-\sum_{k=1}^{j}{\rm{I}}_{k}(\omega)\rb
\right\}$ and $\Theta$ is the measure on $\mathbb{D}_{k}$ induced by $B$.
Furthermore since the sets $\tilde{A}_{j}(\omega)$ have zero Lebesgue measure it suffices to show that
$\int_{0}^{\e_q-S_{k}(\omega)}e^{B'_{s}}ds$
is absolutely continuous on $A_{j}(\omega)$. Indeed it follows because the law of $\int_{0}^{t}e^{B'_{s}}ds$ is absolutely continuous for any $t>0$ and non trivial Brownian motion, see \cite{Yor-01}.
When $B'_{s}=as$ the same claims follows as the measures $\Gamma_{j,\omega}$ are absolutely continuous and hence so is $\int_{0}^{\e_q-S_{k}(\omega)}e^{as}ds$. With this ends the proof of the absolute continuity of the law of ${\rm{I}}_{\Psi_{q}}$.

\noindent Next from \cite[Proposition 3.1.(i)]{Carmona-Petit-Yor-97} for $q>0$ and \cite[Lemma 2.1]{Maulik-Zwart-06} for the case $q=0$ the following equation
\begin{eqnarray} \label{eq:rec_sub1}
\E\left[{\rm{I}}_{\Psi_q}^z\right] &=& -\frac{z}{\Psi_q(z)}\:\E\left[{\rm{I}}_{\Psi_q}^{z-1}\right]
\end{eqnarray}
holds for any $z \in \mathbb{C}$ such that $0<\Re(z)<\beta_q^{*}$,  where we recall from Proposition \ref{prop:mapping}, that for any $q\geq0$, $\beta_q^{*}>0$ is valid. We point out that  all quantities involved are finite since for every $q\geq 0$ for which ${\rm{I}}_{\Psi_{q}}$ is well-defined, we have using \cite[Lemma 2]{Rivero} and a monotone argument, that
\begin{equation}\label{eq:fm}
\Ec{{\rm{I}}^{\rho}_{\Psi_{q}}}<\infty, \quad \textrm{ for all } \rho\in\lb-1,\beta_q^{*}\rb.
\end{equation}
 Thus,  for any $\beta \in (0,\beta_q^{*})$, we have,  for any $-\beta<\Re(z)<\beta_q^{*}-\beta$,
\begin{eqnarray}\label{Th1-2}
\E\left[{\rm{I}}_{\Psi_q}^{z+\beta}\right] &=& -\frac{z+\beta}{\Psi_q(z+\beta)}\:\E\left[{\rm{I}}_{\Psi_q}^{z+\beta-1}\right].
\end{eqnarray}
We note in particular from \eqref{eq:fm} that $\E\left[{\rm{I}}_{\Psi_q}^{\beta}\right]<\infty$.
On the other hand, we have from  Proposition \ref{prop:mapping}, that for any $q\geq0$ and  any $\beta \in (0,\beta_q^{*})$, $\Tb\Psi_q$ is the Laplace exponent of a L\'evy process with a finite negative mean and thus the random variable ${\rm{I}}_{\mathcal{T}_{\beta}\Psi_q}$ is well-defined. We deduce from  \eqref{eq:rec_sub1} and the definition of the transformation $\mathcal{T}_{\beta}$,  that, for any $-\beta<\Re(z)<\beta_q^{*}-\beta$,
\begin{eqnarray}\label{Th1-3}
\nonumber\E\left[{\rm{I}}_{\mathcal{T}_{\beta}\Psi_q}^z\right] &=& -\frac{z}{\mathcal{T}_{\beta}\Psi_q(z)}\:\E\left[{\rm{I}}_{\mathcal{T}_{\beta}\Psi_q}^{z-1}\right]
\\
&=&-\frac{z+\beta}{\Psi_q(z+\beta)}\:\E\left[{\rm{I}}_{\mathcal{T}_{\beta}\Psi_q}^{z-1}\right].
\end{eqnarray}
Next, since the distribution of ${\rm{I}}_{\Psi_q}$ is absolutely continuous, we have that the function $m_{\beta,q}$ given by
\begin{equation}\label{Th1-4}
m_{\beta,q}(x)=\left(\E[{\rm{I}}_{\Psi_{q}}^{\beta}]\right)^{-1}x^{\beta}m_{\Psi_{q}}(x), \text{ a.e. for $x>0$},
\end{equation}
is well-defined and determines a probability density function. We denote by ${\rm{I}}_{\beta,q}$ the random variable with density $m_{\beta,q}$. Then, clearly from \eqref{eq:fm} and \eqref{Th1-2}, the functional equation
\[\E\left[{\rm{I}}_{\beta,q}^{z}\right]=-\frac{z+\beta}{\Psi_q(z+\beta)}\E\left[{\rm{I}}_{\beta,q}^{z-1}\right]=-\frac{z}{\mathcal{T}_{\beta}\Psi_q(z)}\E\left[{\rm{I}}_{\beta,q}^{z-1}\right]\]
holds for $-\beta<\Re(z)<\beta_q^{*}-\beta$ and $\Ec{{\rm{I}}_{\beta,q}^{-1}}<\infty$ and $\Ec{{\rm{I}}_{\beta,q}^{\delta}}<\infty$, for some $\delta>0$. Thanks to the existence of these moments we can use  \cite[Lemma 4.4]{Pardo-Patie-Savov} to deduce that in the notation of \cite{Pardo-Patie-Savov},  $\mathcal{L}m_{\beta,q}(x)=0$ a.e. Moreover, as $\Ec{{\rm{I}}_{\beta,q}^{-1}}<\infty$ we can apply \cite[Theorem 3.7]{Pardo-Patie-Savov} to get in fact that \eqref{Th1-1} holds. Indeed, otherwise, both the law of ${\rm{I}}_{\mathcal{T_{\beta}}\Psi_{q}}$ and ${\rm{I}}_{\beta,q}$ will be different stationary measures of the generalized Ornstein-Uhlenbeck process associated to the \LLP with exponent $\mathcal{T_{\beta}}\Psi_{q}$ as defined  in \cite[Theorem 3.7]{Pardo-Patie-Savov} which is impossible. The proof of the last claim follows  readily from the limit $\lim_{\beta \to 0} \E\left[{\rm{I}}_{\Psi_q}^{\beta}\right]=1$ combined with identity \eqref{Th1-1}.
\end{proof}

The next result is in the spirit of the result of \cite[Theorem 1]{Pardo-Patie-Savov} in the case ${\bf{E_+}}$ and thus can be seen as its extension.
\begin{proposition} \label{lem:me}
Let us assume that ${\bf{T_{\beta^+}}}$ holds and $e^{\beta y}\mu_+^q(dy) \in \mathcal{P}$, for some $\beta \in (0,\beta^+)$. Then,
\begin{equation}\label{lem:me1}
{\rm{I}}_{\Psi_q}\stackrel{d}={\rm{I}}_{\phi_{q_-}}\times {\rm{I}}_{\psi^{q_+}},
\end{equation}
where $\psi^{q_+}(z)=z\phi_{q_+}(z)$.
\end{proposition}
\begin{proof}
 First, from  Proposition \ref{prop:mapping} we get that, for any $\beta \in (0,\beta_q^{*})$,  $\Tb\Psi_q$ is the Laplace exponent of a L\'evy process with a finite negative mean and thus  ${\rm{I}}_{\Tb\Psi_q}$ is well-defined. Next, $\bf{T}_{\beta^+}$ trivially holds for $\phi_{q_-}$ since it corresponds to the Laplace exponent of a negative of a subordinator.  Clearly $e^{\beta y}\mu_{q_{+}}(dy) \in \mathcal{P}$ implies $\mu_{q_{+}}(dy) \in \mathcal{P}$, thus $\psi^{q_+}$ is the Laplace exponent of an unkilled spectrally positive L\'evy process whose tail of the L\'evy measure has the form $\overline{\Pi}^{q_{+}}(y)dy=\mu_{q_+}(dy),y>0,$ see \cite[Lemma 4.3]{Pardo-Patie-Savov}. Therefore, as $e^{\beta y}\overline{\Pi}^{q_+}(y)dy=e^{\beta y}\mu_{q_{+}}(dy) \in \mathcal{P}$ and  $s\mapsto \phi_{q_+}(s)$  is well-defined on $(-\infty,\beta^+)$ which implies that $|\psi^{q_+}(s)|<+\infty$, for  any $s\in (-\infty,\beta^+)$,  we have that  $\psi^{q_+}$ satisfies the condition $\bf{T}_{\beta^+}$. Thus $\Tb\psi
^{q_+}(s)=\frac{s}{s+\beta}\psi^{q_+}(s+\beta)=s\phi_{q_+}(s+\beta)$ is the Laplace exponent of a proper spectrally positive L\'evy process with a finite negative mean $\phi_{q_+}(\beta)$. Next, since $e^{\beta y}\mu_{q_+}(dy) \in \mathcal{P}$, we have that the unkilled L\'evy process with Laplace exponent $\Tb\Psi_q$ satisfies the condition ${\bf{ E_+}}$ of \cite[Theorem 1]{Pardo-Patie-Savov}. From \eqref{eq:cp} of Proposition \ref{prop:twh} we deduce that
\[\frac{\Tb\Psi_q(z)}{z}=-\frac{z\phi_{q_+}(z+\beta)}{z}\frac{\Tb\phi_{q_-}(z)}{z}=-\frac{\Tb \psi^{q_+}(z)}{z}\frac{\Tb\phi_{q_-}(z)}{z}\]
  and from \cite[Theorem 1]{Pardo-Patie-Savov} that
\begin{equation}\label{eq:itb}
{\rm{I}}_{\Tb\Psi_q}\stackrel{d}={\rm{I}}_{\Tb\phi_{q_-}}\times {\rm{I}}_{\Tb \psi^{q_+}}.
\end{equation}
Then, from Theorem \ref{Th1}, we get that \[\lim_{\beta \to 0}{\rm{I}}_{\Tb\Psi_q}\stackrel{d}={\rm{I}}_{\Psi_q}, \:\lim_{\beta \to 0}{\rm{I}}_{\Tb\phi_{q_-}}\stackrel{d}={\rm{I}}_{\phi_{q_-}} \textrm{ and } \lim_{\beta \to 0}{\rm{I}}_{\Tb\psi^{q_+}}\stackrel{d}={\rm{I}}_{\psi^{q_+}},\] which completes the proof.
\end{proof}

Next we provide a killed version of the Vigon's equation amicale,  see \cite[Theorem 16]{Doney}.
\begin{proposition} \label{lem:ak}
Let  us assume that ${\bf {T_{\beta^+}}}$ holds.
\begin{enumerate}
\item Then, we have
\begin{equation}\label{Vigon1}
\overline{\mu}_{q_+}(y)=\int_0^{\infty}\overline{\Pi}_+(r+y)\mathcal{U}^{(q)}_-(dr), \: y>0,
\end{equation}
where $ \int_0^{\infty}e^{-sy}\mathcal{U}^{(q)}_-(dy) = \frac{1}{\phi_{q_-}(z)}$.
\item  Moreover $\phi_{q_+}$ satisfies the condition ${\bf{T_{\beta^+}}}$. Finally,  if for some  $\beta \in (0,\beta^+)$, $e^{\beta y}\Pi_{+}(dy)\in\mathcal{P}$ then $e^{\beta y}{\mu}_{q_+}(dy)\in \mathcal{P}$.
\end{enumerate}
\end{proposition}
\begin{proof} We consider only the case when $q>0$ since when $q=0$ we are in the setting of the classical Vigon's equation amicale. Next, the latter applied to the unkilled L\'evy process $\xi^{(\beta,q)}$ as defined in Proposition \ref{prop:mapping} with $\beta \in (0,\beta_q^{*})$,   yields, with the obvious notation,
\begin{equation}\label{eq:ek}
 \overline{\mu}^{\beta}_{q_+}(y)=\int_0^{\infty}\overline{\Pi}^{\beta}_+(y+r)\mathcal{U}^{(\beta,q)}_-(dr),\end{equation}
where, from \eqref{eq:ct_t}, we have \begin{equation} \label{eq:ttk}
\overline{\Pi}^{\beta}_+(y)=\int_y^{\infty}\Pi^{\beta}(dr)\mathbb{I}_{\{y>0\}}= \int_y^{\infty}e^{\beta
r}\left(\Pi(dr)-\beta\overline{\Pi}_{+}(r)dr\right)\mathbb{I}_{\{y>0\}}\end{equation}
and from Proposition \ref{prop:twh} and \cite[p.~74]{Bertoin-96}
 \[\int_0^{\infty}e^{-zy}\mathcal{U}^{(\beta,q)}_-(dy) = \frac{1}{\mathcal{T}_{\beta}\phi_{q_-}(z)}= \frac{1}{\frac{z}{z+\beta}\phi_{q_-}(z+\beta)}=\frac{1}{\phi_{q_-}(z+\beta)}+\beta\frac{1}{z\phi_{q_-}(z+\beta)}.\]
From the latter we immediately deduce by comparing the Laplace transforms that
\begin{equation}\label{Mladen2}
\mathcal{U}^{(\beta,q)}_-(dy) = e^{-\beta y}\mathcal{U}^{(q)}_-(dy)+\beta\int_0^ye^{-\beta r}\mathcal{U}^{(q)}_-(dr)dy.
\end{equation}
Plugging \eqref{Mladen2} into \eqref{eq:ek}, we get, for all $y>0$,
\begin{align*}
& \overline{\mu}^{\beta}_{q_+}(y)=\int_{0}^{\infty}\overline{\Pi}^{\beta}_+(y+r)e^{-\beta r}\mathcal{U}^{(q)}_-(dr)+\beta \int_{0}^{\infty}\overline{\Pi}^{\beta}_+(y+r)\int_0^re^{-\beta v}\mathcal{U}^{(q)}_-(dv)dr.
\end{align*}
Next, we have, using identity \eqref{eq:ct_t} and the fact that condition ${\bf{T_{\beta_{+}}}}$ holds, the
  existence of a constant $C>0$ such that, for all $\beta$ small enough,
 \[ \int_{y}^{\infty}{\overline{\Pi}}^{\beta}_+(r)dr\leq \int_{y}^{\infty}\int_{r}^{\infty}e^{\beta v}\Pi_+(dv)dr=\int_{y}^{\infty}re^{\beta r}\Pi_+(dr)\leq C.\]
Using this inequality and recalling that, for any $q>0$, $\mathcal{U}^{(q)}_-$ is a positive finite measure on $\R^+$, as a potential measure of a negative of a killed subordinator, that  is a transient Markov process, we obtain, with $\overline{\mathcal{U}}^{(q)}=\mathcal{U}^{(q)}_-(0,\infty)$, that
\begin{eqnarray*}
\int_{0}^{\infty}\overline{\Pi}^{\beta}_+(y+r)\int_0^re^{-\beta v}\mathcal{U}^{(q)}_-(dv)dr %&=&\int_{0}^{\infty}e^{-\beta v}\int_v^\infty \overline{\Pi}^{\beta}_+(y+r) dr \mathcal{U}^{(q)}_-(dv)\\
&\leq&  \overline{\mathcal{U}}^{(q)}\int_{y}^{\infty}{\overline{\Pi}}^{\beta}_+(r)dr\: \leq  \overline{\mathcal{U}}^{(q)}\int_{y}^{\infty}re^{\beta r}\Pi_+(dr)  \leq C_q,
\end{eqnarray*}
%to prove \eqref{Vigon1} it suffices to show that  for a.e. $x>0$, $\lim_{\beta\to 0}\overline{\Pi}^{\beta}_+(y)=\overline{\Pi}_+(y)$. But clearly at
where the constant $C_q>0$ is also uniform for all $\beta$ small enough. This gives us
\begin{equation} \label{eq:ip}
\lim_{\beta\to 0}\overline{\mu}^{\beta}_{q_+}(y)=\lim_{\beta\to 0}\int_{0}^{\infty}\overline{\Pi}^{\beta}_+(y+r)e^{-\beta r}\mathcal{U}^{(q)}_-(dr).\end{equation}
Since for all $\beta$ small enough $ \int_{y}^{\infty}e^{\beta r}\overline{\Pi}_+(r)dr\leq C_{1}$, with $C_{1}>0$,   we have, from \eqref{eq:ttk}, at the points of continuity of $\Pi_{+}(dy)$,   that
\begin{align*}
&\lim_{\beta\to 0}\overline{\Pi}^{\beta}_+(y)=\lim_{\beta\to 0}\lb \int_{y}^{\infty}e^{\beta r}\Pi_{+}(dr)-\beta\int_{y}^{\infty}e^{\beta r}\overline{\Pi}_+(r)dr\rb=\overline{\Pi}_+(y).
\end{align*}
Since $\mathcal{U}_-^{(q)}$ defines a finite measure, we conclude from \eqref{eq:ip} that $\lim_{\beta\to 0}\overline{\mu}^{\beta}_{q_+}(y)=\overline{\mu}^{q}_+(y)$ and hence \eqref{Vigon1} holds.
Next, the fact that the mapping $s\mapsto \phi_{q_+}(s)$ is well defined on $(0,\beta_{+})$ follows readily from \cite[Lemma 4.2]{Pardo-Patie-Savov} since $\Psi_q$ satisfies the condition ${\bf{T_{\beta^+}}}$. Then, for any $0<\beta<\beta^+$, \eqref{Vigon1} gives us that
\[e^{\beta y}\overline{\mu}_{q_+}(y)=\int_0^{\infty}e^{\beta (y+r)}\overline{\Pi}_+(y+r)e^{-\beta r}\mathcal{U}^{(q)}_-(dr).\]
The claim that $e^{\beta y}\overline{\mu}_{q_+}(y)\in \mathcal{P}$ now follows from the fact that  for every fixed $r>0$, the mapping $y\mapsto e^{\beta (y+r)}\overline{\Pi}_+(y+r)$ is non-increasing on $\R^+$.  Hence   $\phi_{q_+}$ also satisfies ${\bf T_{\beta_{+}}}$. Assume now that $e^{\beta y}\Pi_{+}(dy)\in\mathcal{P}$, then one may write $\Pi_+(dy)=\pi_{+}(y)dy$ and the equation
\[e^{\beta  y}\mu_{q_+}(dy)=\int_0^{\infty}e^{\beta(y+r)}\pi_+(y+r)e^{-\beta r}\mathcal{U}^{(q)}_-(dr)dy,\]
which is a differentiated  version of \eqref{Vigon1} shows that $e^{\beta y}{\mu}_{q_+}(dy)\in \mathcal{P}$. To rigorously justify the exchange of differentiation and integration in the differentiated version above note that under ${\bf{T_{\beta^+}}}$ the differentiated version is clearly valid if $q>0$ since  $\mathcal{U}^{(q)}_-$ defines a finite measure. Moreover, when $q=0$ and $\beta>0$, $e^{-\beta r}\mathcal{U}_-(dr)$ is a finite measure due to the sublinearity of the potential function $\mathcal{U}_-((0,r))$, see \cite[p 74]{Bertoin-96}. Finally when both $q=0$ and $\beta=0$ the differentiated version follows from \cite[Lemma 4.11]{Pardo-Patie-Savov}.  %If in addition $\beta \in (0,\beta_q^{*})$ then the second identity in \eqref{Esscher111} shows that $\mu^{q,\beta}_+$ is a Esscher transform of $\mu_{q_+}$ and thus $\mu^{q,\beta}_+(dy)=e^{\beta y}\mu_{q_+}(dy)$ and hence $\mu^{q,\beta}_+(dy)\in \mathcal{P}$ which completes the proof of the second claim.
\end{proof}
In order to complete the proof of Theorem \ref{MainTheorem} in the case ${\bf P}+$ we will resort to some approximation procedures for which we need the following results.
\begin{lemma}\label{lem:ce}
\begin{enumerate}[(a)]
\item Let $(\phi_-^{(n)})_{n\geq1}$ be a sequence of Laplace exponents of negative of possibly killed subordinators. Assume that for all $s\geq0$, $\lim_{n\to \infty}\phi_-^{(n)}(s)=\phi_-(s)$, where $\phi_-$ is the Laplace exponent of a negative of a possibly killed subordinator. Then
\begin{equation*}
\lim_{n\to\infty}{\rm{I}}_{\phi_-^{(n)}}\stackrel{d}={\rm{I}}_{\phi_-}.
\end{equation*}
\item Let $(\Psi^{(n)})_{n\geq1}$ be a sequence of characteristic exponents of \LLPs  such that, for all $z\in i\R$,
\begin{equation} \label{eq:cpl}
 \lim_{n\to\infty}\Psi^{(n)}(z)= \Psi(z),
 \end{equation}
where $\Psi$ is the characteristic exponent of a L\'evy process.  Assume further that for all $n\geq1$, $\Psi^{(n)}(0)= \Psi(0)=0$. Then, for all $q>0$,
\begin{equation*}
\lim_{n\to\infty}{\rm{I}}_{\Psi_q^{(n)}}\stackrel{d}={\rm{I}}_{\Psi_q}.
\end{equation*}
  \end{enumerate}
\end{lemma}
\begin{remark}
A case similar to $(a)$ was treated in Lemma 4.8 in \cite{Pardo-Patie-Savov}. However there it is assumed that the subordinators are proper. Note that case $(b)$ is far simpler than Lemma 4.8 in \cite{Pardo-Patie-Savov} as we are strictly in the killed case and the exponential functional up to a finite time horizon is continuous in the Skorohod topology.
\end{remark}
\begin{proof}
 First we use the fact that the law of the exponential functional of a negative of a possibly killed subordinator is moment determinate. More specifically, Carmona et al.~\cite{Carmona-Petit-Yor-97}, showed, writing $\E\left[{\rm{I}}_{\phi_-^{(n)}}^m\right] =M_m^{(n)}$, that
\begin{eqnarray}\label{eq:ms}
M_m^{(n)} &=&\frac{\Gamma(m+1)}{\prod_{k=1}^{m}-\phi_-^{(n)}(k)}, \: m=1,2,\ldots\,\,.
\end{eqnarray}
From the convergence of the Laplace exponents, we deduce that, for all integers $m\geq 1$, $\lim_{n\to \infty}M_m^{(n)} =\frac{\Gamma(m+1)}{\prod_{k=1}^{m}-\phi_-(k)}$, which is the $m$-th moment of the exponential functional ${\rm{I}}_{\phi_-}$. Item (a) follows then from \cite[Examples (b) p.269]{Feller-71}.
Next, \eqref{eq:cpl} combined with $\Psi^{(n)}(0)= \Psi(0)=0$, implies that the corresponding sequence of L\'evy processes $\lb\xi^{(n)}\rb_{n\geq 1}$ converges in distribution to a L\'evy process $\xi$. Using Skorohod-Dudley's theorem, we assume  that the convergence holds a.s.~on the Skorohod space $\mathcal{D}((0,\infty))$ and check that, for any $t>0$,
\[\int_0^t e^{\xi^{(n)}_s}ds\stackrel{d}\rightarrow \int_0^t e^{\xi_s}ds.\]
 Then applying Portmanteau's theorem, for any fixed $t,x\geq 0$, we have that
\[\limsup_{n\rightarrow \infty} \P\left(\int_0^t e^{\xi^{(n)}_s}ds \leq x\right)\leq \P\left(\int_0^t e^{\xi_s}ds \leq x\right)\]
\[\liminf_{n\rightarrow \infty} \P\left(\int_0^t e^{\xi^{(n)}_s}ds < x\right)\geq \P\left(\int_0^t e^{\xi_s}ds < x\right).\]
Hence since, for any $q>0$ and $A\subset \R_{+}$,
\[\P\lb {\rm{I}}_{\Psi_{q}}\in A\rb=q\int_{0}^{\infty}e^{-qt}\P\lb \int_{0}^{t}e^{\xi_{s}}ds\in A\rb dt\]
and $qe^{-qt}dt$ defines a finite measure, we have from the reverse Fatou's lemma (resp. Fatou's lemma) that
\[\limsup_{n\rightarrow \infty} \int_0^{\infty}dt q e^{-qt}\P\left(\int_0^t e^{\xi^{(n)}_s}ds \leq x\right)\leq \int_0^{\infty}dt q e^{-qt}\P\left(\int_0^t e^{\xi_s}ds \leq x\right)=\P\lb {\rm{I}}_{\Psi_{q}}\leq x\rb\]
\[\liminf_{n\rightarrow \infty} \int_0^{\infty}dt q e^{-qt}\P\left(\int_0^t e^{\xi^{(n)}_s}ds < x\right)\geq \int_0^{\infty}dt q e^{-qt}\P\left(\int_0^t e^{\xi_s}ds < x\right)=\P\lb {\rm{I}}_{\Psi_{q}}< x\rb.\]
This suffices since from Theorem \ref{Th1} we know that $\P\lb {\rm{I}}_{\Psi_{q}}= x\rb=0$, for all $x\geq 0$.
\end{proof}

\begin{comment}
\begin{lemma} \label{lem:itb}
Let us assume that ${\bf{T_{\beta^+}}}$ holds for some $\beta^+>0$  and  that $\Tb \Psi$ satisfies the condition $({\bf{P+}})$ for some $0<\beta<\beta^+$. Then, for any $q\geq0$,
\begin{equation}\label{eq:itb}
{\rm{I}}_{\Tb\Psi_q}\stackrel{d}={\rm{I}}_{\Tb\phi_{q}}\times {\rm{I}}_{\Tb \psi^{q_+}}
\end{equation}
where $\psi^{q_+}(-u)=-u\phi_{q_+}(-u),u>0$.
\end{lemma}
\begin{proof}
This follows by a similar reasoning as in the previous proof combined with the fact that for any $q>0$, the restrictions on $\R^+$ of the  L\'evy measures of $\Tb\Psi_q$ and $\Tb\Psi$ coincide, see \eqref{eq:cp1},  thus $\Tb\Psi_q$ also satisfies the condition $({\bf{P+}})$.
\end{proof}
\end{comment}
Now, we have all the ingredients to complete  the proof of Theorem \ref{MainTheorem} in the case ${\bf{P+}}$.
Let us consider, for any $\delta>0$,  the L\'evy process $\xi^{(\delta)}=(\xi_t^{(\delta)})_{t\geq0}$, with Laplace exponent $\Psi^{(\delta)}$, constructed from $\xi$ by tilting the positive jumps. More precisely,  we  modify  the L\'evy measure of $\xi$ as follows
  \[\Pi^{(\delta)}(dy)= \Pi(dy)\Id{y<0}+e^{-\delta y}\Pi_+(dy)\]
  and leave the Gaussian coefficient and the linear term untouched.
 From \cite[Theorem 25.17]{Sato-99}, we have that $|\Psi^{(\delta)}(s)|<+\infty$, for any $s\in (0,\delta)$. For $\Psi_q^{(\delta)}$, we define $\beta^{*}_\delta(q) $ as in \eqref{eq:db}
   and choose $\beta$ such that  $0<\beta<\delta\wedge\beta^{*}_\delta(q) =\delta'$. Then, since $\Pi_+(dy)=\pi_+(y)dy \in \mathcal{P}$ the mapping defined on $\R^+$ by
\[ y\mapsto e^{\beta y}\int_y^{\infty} \Pi^{(\delta)}(dr) = e^{(\beta-\delta) y}\int_0^{\infty} e^{-\delta r}\pi_+(r+y)dr\]
is plainly non-increasing. Hence $\Psi^{(\delta)}$ satisfies the condition ${\bf{T_{\delta'}}}$. Moreover,  $e^{\beta y}\Pi_+^{(\delta)}(dy) \in \mathcal{P}$ and hence from the item (2) of  Proposition \ref{lem:ak},  we have with the obvious notation $e^{\beta y}\mu^{(\delta)}_{q_+}(dy) \in \mathcal{P}$.
Thus, the L\'evy process with characteristic exponent  $\Psi_q^{(\delta)}$ satisfies the conditions of Lemma \ref{lem:me} and we deduce that
\begin{equation*}
{\rm{I}}_{\Psi_q^{(\delta)}}\stackrel{d}={\rm{I}}_{\phi^{(\delta)}_{q_-}}\times {\rm{I}}_{ \psi^{(\delta),q_+}},
\end{equation*}
where we have set $ \Psi_q^{(\delta)}(z)=-\phi^{(\delta)}_{q_-}(z)\phi_{q_+}^{(\delta)}(z)$ and $\psi^{(\delta),q_+}(z)=z\phi_{q_+}^{(\delta)}(z)$. Next, since as $\delta\rightarrow 0$, $\Pi^{(\delta)}(dy)\stackrel{v}\rightarrow \Pi(dy)$, where $\stackrel{v}\rightarrow $ stands for the vague convergence, we have that $\lim_{\delta \to 0}\xi^{(\delta)}\stackrel{d}=\xi$, see \cite[Theorem 13.14]{Kallenberg}. Putting $h^{(\delta)}(y)=e^{-\delta y}$ we see that the assumptions of  \cite[Lemma 4.9]{Pardo-Patie-Savov} are satisfied ( note that the only case which \cite[Lemma 4.9]{Pardo-Patie-Savov} does not encompass, i.e. when $q=0$ and $\xi$ does not drift to $-\infty$, is ruled out by our assumptions) and thus we have, for all $s\geq0$,
\begin{eqnarray}
\lim_{\delta \to 0}\phi^{(\delta)}_{q_-}(s) =\phi_{q_-}(s) \textrm{  and  } \lim_{\delta \to 0}\phi_{q_+}^{(\delta)}(-s) =\phi_{q_+}(-s). \label{eq:ca}
\end{eqnarray}
From \eqref{eq:ca} combined with Lemma \ref{lem:ce} (a)  we get that
\begin{eqnarray*}
\lim_{\delta \to 0}{\rm{I}}_{\phi^{(\delta)}_{q_-}}\stackrel{d}={\rm{I}}_{\phi_{q_-}}.
\end{eqnarray*}
Next from \eqref{eq:ca}, we deduce that for any $s\geq 0$, $\lim_{\delta \to 0}\psi^{(\delta),q_{+}}(-s) =\lim_{\delta \to 0}-s\phi_{q_+}^{(\delta)}(-s) =-s\phi_{q_+}(-s)= \psi^{q_+}(-s)$ and $\lim_{\delta \to 0}(\psi^{(\delta),q_{+}})'(0^-) =\lim_{\delta \to 0}\phi_{q_+}^{(\delta)}(0)=\phi_{q_+}(0)=(\psi^{q_{+}})'(0^-)$. Thus, we can apply   \cite[Lemma 4.8 (a)]{Pardo-Patie-Savov} to get
\begin{eqnarray*}
\lim_{\delta \to 0}{\rm{I}}_{\psi^{(\delta),q_{+}}}\stackrel{d}={\rm{I}}_{\psi^{q_+}}.
\end{eqnarray*}
Finally, since $\lim_{\delta \to 0}\xi^{(\delta)}\stackrel{d}=\xi$,  Lemma \ref{lem:ce} (b)  implies, when  $q>0$, that
\begin{eqnarray*}
\lim_{\delta \to 0}{\rm{I}}_{\Psi^{(\delta)}_{q}}\stackrel{d}={\rm{I}}_{\Psi_{q}}
\end{eqnarray*}
and the case when $q>0$ is finished. 
When $q=0$ due to the considerations above we have already shown that
\[\lim_{\delta\to 0}{\rm{I}}_{\phi^{(\delta)}_{q_-}}\times {\rm{I}}_{ \psi^{(\delta),q_+}}={\rm{I}}_{\phi_{q_-}}\times {\rm{I}}_{ \psi^{q_+}}.\]
It remains to show that ${\rm{I}}_{\Psi^{(\delta)}}\stackrel{d}\rightarrow{\rm{I}}_{\Psi}$, as $\delta\rightarrow 0$. From the construction of $\xi^{(\delta)}$ we can write
\[\xi_t=\xi_t^{(\delta)}+\tilde{\xi}_t^{(\delta)}, \: t\geq0,\]
where $\tilde{\xi}^{(\delta)}=(\tilde{\xi}_t^{(\delta)})_{t\geq0}$ is a subordinator with zero drift and \LL measure $(1-e^{-\delta y})\Pi(dy)\mathbb{I}_{\{y>0\}}$ which is taken independent of $\xi^{(\delta)}$. Therefore $\xi_{t}\geq \xi^{(\delta)}_{t}$, for all $t\geq 0$, and hence we conclude that
$\lim_{\delta \to 0}{\rm{I}}_{\Psi^{(\delta)}}\stackrel{d}={\rm{I}}_{\Psi}$
from the monotone convergence theorem. This completes the proof of Theorem \ref{MainTheorem} in the case  ${\bf{P+}}$.
%\begin{remark}
%Note that we did not have to show that $\Pi_+ \in \Ph$ implies that $\mu_+^q \in \Ph$. But actually it should come up as a consequence.
%\end{remark}
\begin{comment} {\bf I am not sure that this approach works anymore...}
  On the other hand, recalling relations \eqref{eq:ct_t} and \eqref{eq:cp1} we see that that the L\'evy measure associated to the characteristic exponent $\Tb\Psi^{(\delta)}_q$ takes the form, when restricted to the positive half-line,
\begin{equation}\label{Mladen1}
 \Pi^{(\delta,\beta)}(dy)\Id{y>0}= e^{\beta y}\left(\Pi^{(\delta)}(dy)-\beta \int_y^{\infty}\Pi^{(\delta)}(dz)dy\right)\Id{y>0}.
 \end{equation}
Then we deduce that if $\Pi_+ \in \Ph$ then, since $0<\beta<\delta$, clearly $\Pi^{(\delta,\beta)}(dy)\Id{y>0} \in \Ph$. Indeed the first term in \eqref{Mladen1} has a non-increasing density as $\beta<\delta$, whereas the density of the second can be rewritten as
\[\beta e^{(\beta-\delta)y}\int_0^{\infty}e^{-\delta z}\Pi(dz+y)\Id{y>0}\]
and it is clearly a non-increasing function as $\Pi_+ \in \Ph$.
\end{comment}

\subsection{The case ${\bf{P^q_{\pm}}}$}
Since the case $q=0$ was treated in \cite{Pardo-Patie-Savov}, we assume in the sequel that $q>0$.
In what follows, we provide a necessary condition on the L\'evy measures of the characteristic exponent of bivariate subordinators in order that they correspond to the Wiener-Hopf factors of a killed L\'evy process. We mention that Vigon \cite{Vigon} provides such a criteria for proper L\'evy processes and our condition relies heavily on his approach.
%Vigon gave a necessary condition for on the L\'evy measure of La that in the Wiener-Hopf factorization in space
\begin{lemma} \label{lem:pk} \label{lem:t2wh}
Let us consider $\phi_{q_+}$ and $\phi_{q_-}$ as defined in Proposition \ref{prop:twh}. Assume  that $\mu_{q_+} \in \mathcal{P}$ and $ \mu_{q_-} \in \mathcal{P}$ with $q=q_+ q_->0$.
 \begin{enumerate}
 \item There exists a characteristic exponent of a killed L\'evy process  $\Psi_q$ such that
\[\Psi_q (z) = - \Phi_+(q,-z)\Phi_-(q,z)=-\phi_{q_+}(z)\phi_{q_-}(z).\]
\item If in addition for any $0<\beta<\beta_+$, for some $\beta_+>0$, $-\infty<\phi_{q_+}(\beta)<0$ and $e^{\beta y}\mu_{q_+}(dy) \in \mathcal{P}$, then  $\Psi_q$ satisfies the condition ${\bf{T_{\beta^+}}}$.
    \end{enumerate}
\end{lemma}
\begin{proof}
From Proposition \ref{prop:twh}, writing $\phi_{q_\pm}(z)=\phi_{\pm}(z)-q_{\pm}$, we observe that
\begin{eqnarray*}
-\Phi_+(q,-z)\Phi_-(q,z) &=& -\phi_{q_+}(z)\phi_{q_-}(z)= -(\phi_+(z)-q_+)\phi_-(z)+q_-\phi_+(z)-q_+q_-.
\end{eqnarray*}
Then, from Vigon's philantropy theory, we know that $-(\phi_+(z)-q_+)\phi_-(z)$ is the characteristic exponent of an unkilled  L\'evy process that drifts to $-\infty$. It is also clear that $q_-\phi_+(z) $  is  the characteristic exponent of an unkilled subordinator. From the inequality $q_+q_->0$ we complete the proof of the first item.
Next, from the form of $\phi_{q_-}$ in Proposition \ref{prop:twh} and carefully using the same techniques as in deriving \eqref{eq:ct_t} we deduce that $\Tb\phi_{q_-}$ is the Laplace exponent of a negative of an unkilled subordinator whose  \LL measure has the form ${\mu}^{\beta}_{q_-}(dy)=e^{-\beta y}({\mu}_{q_-}(dy)+\beta \overline{\mu}_{q_-}(y)dy)$. Similarly, due to our assumption, i.e.~$-\infty<\phi_{q_+}(\beta)<0$, the mapping $s\mapsto \phi_{q_+}(s+\beta)$ is the Laplace exponent of a killed subordinator with \LL measure  $e^{\beta y}\mu_{q_+}(dy)$. As $ \mu_{q_-} \in \mathcal{P}$, we easily check that ${\mu}^{\beta}_{q_-}(dy) \in \mathcal{P}$ and since, by assumption, $e^{\beta y}\mu_{q_+}(dy) \in \mathcal{P}$, we have from the first item  that there exists a characteristic exponent $\Psi^{\beta}$, of an unkilled L\'evy process drifting to $-\infty$, which is defined by
\[\Psi^{\beta}_q (z)=-\phi_{q_+}(z+\beta) \Tb\phi_{q_-}(z)= -\phi_{q_+}(z+\beta)\frac{z}{z+\beta}\phi_{q_-}(z+\beta) .\]
Moreover, as
\[\mathcal{T}_{\beta}\Psi_q(z)=\frac{z}{z+\beta}\Psi_q(z+\beta)= -\phi_{q_+}(z+\beta)\frac{z}{z+\beta}\phi_{q_-}(z+\beta)= -\phi_{q_+}(z+\beta)\Tb\phi_{q_-}(z),\]
we deduce, by means of an uniqueness argument, that $\Tb\Psi_{q}=\Psi^{\beta}_q$. Then, by the mere definition of condition ${\bf{T_{\beta^+}}}$ we check that $\Psi_q$ satisfies condition ${\bf{T_{\beta^+}}}$.
 \end{proof}
\begin{comment}
We shall also need the following  result which  is purely technical.
\begin{lemma}\label{lem:twh}
 Let $\nu_{+}$ be a positive measure on $\R_{+}$ such that $\int_{0}^{\infty}(y^{2}\wedge1)\nu_+(dy)<\infty$. Assume that for some $\beta>0$, $e^{\beta y}\nu_+(dy)\in \mathcal{P}$ then $e^{\beta y}\bar{\nu}_+(y)dy\in \mathcal{P}$ where $\bar\nu_+(y)=\int_{y}^{\infty}\nu_+(dr)$.
\end{lemma}
\begin{proof}
Since $e^{\beta y}\nu_+(dy)\in \mathcal{P}$,  the measure  $\nu_+$ admits a density which we denote by $v_+$. Then
\[e^{\beta y}\bar{\nu}_+(y)=\int_0^{\infty}e^{\beta(y+r)}v_+(y+r)e^{-\beta r}dr.\]
Since for every fixed $r>0$, the mapping $y\mapsto e^{\beta(y+r)}v_+(y+r)$ is non-increasing on $\R^+$,  the claim (1) follows.
\end{proof}
\end{comment}
We are ready to complete the proof of Theorem \ref{MainTheorem}. First, we set, for any $\delta>0$,
\begin{equation} \label{eq:phi_delta}
\phi_{q_+}^{(\delta)}(z)=\phi_{q_+}(z-\delta)-\phi_{q_+}(-\delta)+\phi_{q_+}(0).
\end{equation}
This is the Laplace exponent of a subordinator with drift $\delta_+$, killing rate $-\phi_{q_+}(0)=q_+>0$ and L\'evy measure  ${\mu}^{(\delta)}_{q_+}(dy)=e^{-\delta y}\mu_{q_+}(dy).$ Next we choose $\delta>0$ so small that $\phi_{q_+}^{(\delta)}(\delta)<0$.
 Since, by assumption $\mu_{q_\pm} \in \mathcal{P}$ plainly  ${\mu}^{(\delta)}_{q_+} \in \mathcal{P}$, and  thus according to item (1) of Lemma \ref{lem:pk}, there exists a characteristic exponent $\Psi_q^{(\delta)}$ of a killed L\'evy process  such that
\begin{equation}\label{eq:psi_delta}
\Psi^{(\delta)}_q(z)= -\phi_{q_+}^{(\delta)}(z)\phi_{q_-}(z).
\end{equation}
Moreover since we have that $|\phi_{q_+}^{(\delta)} (s)|<+\infty$, for any $s<\delta$, we get from \cite[Lemma 4.2]{Pardo-Patie-Savov} that $|\Psi^{(\delta)}_q(s)|<+\infty$, for any $0<s<\delta$.  Also, since   $\phi_{q_+}^{(\delta)}$ is increasing on $(-\infty, \delta)$, we get from our choice of $\delta$ that, for any $0<\beta<\delta $,  $-\infty<\phi_{q_+}^{(\delta)}(\beta)<0$. As for any $0<\beta<\delta $, $e^{\beta y}{\mu}^{(\delta)}_{q_+}(dy) \in \mathcal{P}$, we deduce from item (2) of Lemma \ref{lem:t2wh}, that $\Psi_q^{(\delta)}$ satisfies the condition ${\bf{T_\delta}}$. Hence, we can apply Proposition \ref{lem:me} to get the identity
\begin{equation*}
{\rm{I}}_{\Psi_q^{(\delta)}}\stackrel{d}={\rm{I}}_{\phi_{q_-}}\times {\rm{I}}_{\psi^{(\delta),q_{+}}},
\end{equation*}
where $\psi^{(\delta),q_{+}}(z)=z\phi_{q_+}^{(\delta)}(z)$.  Next, on the one hand, we have, from \eqref{eq:phi_delta}, that for any $s\geq0,$ $\lim_{\delta \to 0}\phi_{q_+}^{(\delta)}(s)=\phi_{q_+}(s)$ and thus $\lim_{\delta \to 0}\psi^{(\delta),q_+}(s)=\psi^{q_+}(s)$ together with $\lim_{\delta \to 0}(\psi^{(\delta),q_+})'(0^-)=\lim_{\delta \to 0}\phi^{(\delta)}_{q_+}(0)=\phi_{q_+}(0)=(\psi^{q_+})'(0^-)$. Thus, we can use \cite[Lemma 4.8(a)]{Pardo-Patie-Savov} to get $
\lim_{\delta \to 0}{\rm{I}}_{\psi^{(\delta),q_+}}\stackrel{d}={\rm{I}}_{\psi^{q_{+}}}$. On the other hand, we deduce  from \eqref{eq:psi_delta} that, for any $z \in i\R$, $\lim_{\delta \to 0}\Psi^{(\delta)}_q(z)=\Psi_{q}(z)$ and for any $\delta\geq0$, $\Psi^{(\delta)}_q(0)=\Psi_{q}(0)$. Hence, from from Lemma \ref{lem:ce} (b), we have
$
\lim_{\delta \to 0}{\rm{I}}_{\Psi^{(\delta)}_{q}}\stackrel{d}={\rm{I}}_{\Psi_q}$,
 which completes the proof of Theorem \ref{MainTheorem}.
\section{Proof of the corollaries and some examples}\label{sec:proof-cor}

\subsection{Proof of Corollary \ref{cor:main}}
From the Wiener-Hopf factorization \eqref{eq:wh1} and the assumptions we have that $-\infty<\Psi_q(-1)=-\phi_{q_+}(-1)\phi_{q_-}(-1)\leq 0$. Then we get from \cite[Lemma 4.1]{Pardo-Patie-Savov}  that the mapping $s\mapsto \phi_{q_-}(s)$ is well-defined on $ [-1,\infty)$ and since $\phi_{q_+}(-1)<0$, we conclude that $\phi_{q_-}(-1)\leq 0$. Thus,  $\tilde{\phi}_{q_-}(s)=\phi_{q_-}(s-1)$ is a Laplace of a negative of a possibly killed subordinator and so  $\mathcal{T}_1\tilde{\phi}_{q_-}$ is the Laplace exponent of a negative of a proper subordinator.
From \eqref{eq:ms}, we have, for $m=1,2,\ldots,$
\begin{eqnarray*}
\E\left[{\rm{I}}_{\phi_{q_-}}^m\right] &=&\frac{\Gamma(m+1)}{\prod_{k=1}^{m}-\phi_{q_-}(k)} =\frac{1}{m+1}\frac{\Gamma(m+1)}{\prod_{k=1}^{m}-\frac{k}{k+1}\tilde{\phi}_{q_-}(k+1)}=\frac{1}{m+1}\frac{\Gamma(m+1)}{\prod_{k=1}^{m}-\mathcal{T}_1\tilde{\phi}_{q_-}(k)}.
\end{eqnarray*}
By moment identification and moment determinacy of ${\rm{I}}_{\phi_{q_-}}$, see \cite{Carmona-Petit-Yor-97}, we deduce that
\begin{equation} \label{eq:us}
{\rm{I}}_{\phi_{q_-}} \stackrel{d}= U  \times {\rm{I}}_{\mathcal{T}_1\tilde{\phi}_{q_-}},
\end{equation}
where $U$ stands for an uniform random variable on $(0,1)$.  Thus,  from Khintchine  Theorem, see e.g.~\cite[Theorem p.158]{Feller-71}, we have that  $m_{\phi_{q_-}}$ is non-increasing on $\R^+$.  We also get, from \eqref{eq:us}, that 
\[ 
{m}_{\phi_{q_-}}(x)= \int_x^{\infty}{m}_{\mathcal{T}_1\tilde{\phi}_{q_-}}(y)dy/y,
 \]
which combined with  \eqref{Th1-1} and \eqref{eq:ms} yields  ${m}_{\phi_{q_-}}(0)= -\tilde{\phi}_{q_-}(1)=-\phi_{q_-}(0)>0$ since when $q= 0$ we assume that $\xi$ drifts to $-\infty$ and hence the descending ladder height process is the negative of a killed subordinator.
Since we also suppose  that either one of the two conditions of Theorem \ref{MainTheorem} applies, we conclude that
\begin{equation}
{\rm{I}}_{\Psi_q} \stackrel{d}= U  \times {\rm{I}}_{\mathcal{T}_1\tilde{\phi}_{q_-}}\times {\rm{I}}_{\psi^{q_+}}=U\times V
\end{equation}
which gives that $m_{\Psi_q}$ is non-increasing on $\R^+$ and hence a.e. differentiable on $\R^+$. Moreover, since 
\[ 
m_{\Psi_q}(x)=\int_0^{\infty}{m}_{\phi_{q_-}}(x/y){m}_{\psi^{q_+}}(y)dy/y=\int_{x}^{\infty}{m}_{V}(y)dy/y,
 \]
we deduce from the discussion above and  an argument of dominated converge that
\begin{eqnarray*}
m_{\Psi_q}(0)&=&-\phi_{q_-}(0)\int_0^{\infty}{m}_{\psi^{q_+}}(y)dy/y \\
&=&\phi_{q_-}(0)\phi_{q_+}(0)=q
 \end{eqnarray*}
where the last line follows from \eqref{eq:msn}. To prove the claim of continuity in item (i) note that from the second integral representation $m_{\Psi_q}(x)$ is continuous.

\noindent In order to prove the first statement of item (ii) we show that, for any $q>0$, we have the following factorization
\begin{equation}\label{eq:hye}
{\rm{I}}_{\Psi_q}\stackrel{d}= \e_{1} \times {\rm{I}}_{\psi^{q_+}},
\end{equation}
where  $\psi^{q_+} (z)=z\Psi_{q}(z) $. Indeed, this  identity follows readily from Theorem \ref{MainTheorem}, since, in this case, $\mu_{q_+} \in \Ph$, $\phi_{q_-}(z) \equiv 1$ and thus ${\rm{I}}_{\phi_{q_-}} = \int_0^{\e_{1}}e^{0}ds=\e_{1}$. Thus,  ${\rm{I}}_{\Psi_q}$ is a mixture of exponential distributions and the complete monotonicity property of  its density follows from \cite[Theorem 53.2]{Sato-99}.  Moreover, from \eqref{eq:hye}, we deduce that
\[m_{\Psi_q}(x)=\int_0^{\infty}e^{-x/y}m_{\psi^{q_{+}}}(y)dy/y\]
and for any $x<\lim_{s \to \infty}-s\Psi_q(-s)=1/b>0$, we get
\begin{eqnarray*}
m_{\Psi_q}(x)&=& \sum_{n=0}^{\infty}\frac{1}{n!}(-x)^n\int_0^{\infty}y^{-n-1}m_{\psi^{q_{+}}}(y)dy\\
&=&  q\left(1+\sum_{n=1}^{\infty}\frac{\prod_{k=1}^{n}-\Psi_{q}(-k)}{n!}(-x)^n\right),
\end{eqnarray*}
where we have used an  argument of dominated convergence and \eqref{eq:msn}. Next, assume that $b>0$ and thus the previous power series defines a function analytical  on the disc of radius $b$. Since the mapping $x \mapsto m_{\Psi_q}(x)$ is  the Laplace  transform of some positive measure, its first singularity occurs on the negative real line, see e.g.~\cite[Chap.~2]{Widder}, which means at the point $-b$. Following the proof of \cite[Proposition 2.1]{Patie-abs-08}, we can then apply the Euler transform, see e.g.~\cite{Prodinger}, to obtain the power series representation \eqref{eq:pse} which actually defines an analytical function on the half-plane $\Re(z)>-(2b)^{-1}$.   The proof of the claims of (ii) is completed after observing from the power series representations that $m_{\Psi_q}(0)=q$.
Item (iii) follows easily from the Wiener-Hopf factorization for spectrally positive L\'evy processes which yields the identity
\[\Psi_q(s)=-\frac{\Psi(s)-q}{s+\gamma_q}(-s-\gamma_q).\]
Thus, in this case, we have ${\rm{I}}_{\phi_{q_-}}=\int_0^{\e_{\gamma_q}} e^{-s}ds= 1-e^{-\e_{\gamma_q}}$ which can easily be seen to be a $B^{-1}(1,\gamma_q)$, which provides the factorization from Theorem \ref{MainTheorem}. We complete the proof of this item by recalling that in this case the mapping $s\mapsto \Psi_q(s)$ is well-defined on $\R^-$ and $\mu_{q_+} \in \mathcal{P}$, see e.g. \cite[Remark p. 103]{Vigon}.
Finally, the proof of the item (iv) goes along the lines of the one of \cite[Corollary 2.1]{Pardo-Patie-Savov}.
\subsection{Some illustrative examples}
For this part, we introduce the following notation. For any $\gamma>0$  and $0<\alpha<1$, we write, for any $ s\geq-\gamma$, \[\phi_{\gamma}(s)=(s+\gamma)^{\alpha}.\] We start by considering  the case where $\phi_{q_-}(s)=-\phi_{\gamma}(s),\: s\geq0$, that is, using the notation of Proposition \ref{prop:twh}, $\mu_-^q(dy)=\frac{\alpha}{\Gamma(1-\alpha)} e^{-\gamma y}y^{-\alpha-1}dy, q_-=\gamma^{\alpha}, \delta_-=0$. Since the random variable ${\rm{I}}_{\phi_{q_-}}$ is moment determinate, we easily get, from \eqref{eq:ms}, that, for any $\Re(z)>-1$,
\begin{eqnarray} \label{eq:mpg}
\E\left[{\rm{I}}_{\phi_{q_-}}^z\right] &=&\frac{\Gamma(z+1)\Gamma^{\alpha}(\gamma+1)}{\Gamma^{\alpha}(z+1+\gamma)} .
\end{eqnarray}
Assuming, for sake of simplicity, that $\gamma$ is not an integer, and applying the inverse Mellin transform, see e.g.~\cite[Section 3.4.2]{Paris}, we get
\[m_{{\rm{I}}_{\phi_{q_-}}}(x)=\sum_{n=0}^{\infty}\frac{\Gamma^{\alpha}(\gamma+1)}{ \Gamma^{\alpha}(-n+\gamma)}\frac{(-x)^n}{n!}\]
and the series is easily seen to be absolutely convergent for all $ x>0$. 
From Corollary \ref{cor:main} (i), we deduce that for any $\gamma\geq1$, this series is positive and non-increasing on $\R^+$. We can also check that $m_{{\rm{I}}_{\phi_{q_-}}}(0)=-\phi_{q_-}(0)=\gamma^{\alpha}$.  Moreover, assuming that $\mu_{q_+}(dy) \in \mathcal{P}$, then according to Proposition \ref{lem:pk} there exists a Laplace exponent $\tilde{\Psi}_q$ of a killed L\'evy process such that
\[\tilde{\Psi}_q(z)=\phi_{\gamma}(z) \phi_{q_+}(z)\]
and where we have set $q=\gamma^{\alpha} \phi_{q_+}(0)>0$. As above, we wote that from Corollary \ref{cor:main} (i),  that  for any $\gamma\geq1$, the density $m_{{\rm{I}}_{\tilde{\Psi}_q}}$ is bounded and non-increasing  on $\R^+$. Next, according to the case  ${\bf{P^q_{\pm}}}$ of Theorem \ref{MainTheorem}, we have, writing, $\psi^{q_+}(z)=z\phi_{q_+}(z)$,  that
\begin{eqnarray}
m_{{\rm{I}}_{\tilde{\Psi}_q}}(x)
&=&\int_0^{\infty}\sum_{n=0}^{\infty}\frac{\Gamma^{\alpha}(\gamma+1)}{ \Gamma^{\alpha}(-n+\gamma)}\frac{(-x/y)^n}{n!}m_{\psi^{q_+}}(y)/ydy \nonumber \\
&=&
\sum_{n=0}^{\infty}\frac{\Gamma^{\alpha}(\gamma+1)}{ \Gamma^{\alpha}(-n+\gamma)}\frac{(-x)^n}{n!}\int_0^{\infty}y^{-n-1}m_{\psi^{q_+}}(y)dy \nonumber  \\
&=&-\phi_{q_+}(0)\sum_{n=0}^{\infty}\frac{\Gamma^{\alpha}(\gamma+1)}{ \Gamma^{\alpha}(-n+\gamma)}\frac{\prod_{k=1}^{n}\psi^{q_{+}}(-k)}{n!}\frac{(-x)^n}{n!} \nonumber  \\
&=&-\phi_{q_+}(0)\Gamma^{\alpha}(\gamma+1)\sum_{n=0}^{\infty}\frac{\prod_{k=1}^{n}\phi_{q_+}(-k)}{ \Gamma^{\alpha}(-n+\gamma)}\frac{x^n}{n!} \label{eq:pse1}
\end{eqnarray}
 where the interchange of integration and summation is justified by an argument of dominated convergence under the condition that $x<\lim_{s\to \infty} -s^{1-\alpha} /\phi_{q_+}(-s)$. We easily check that, under this condition,  the density is actually bounded with $m_{{\rm{I}}_{\tilde{\Psi}_q}}(0)=-\phi_{q_+}(0)\gamma^{\alpha}$.
  
\noindent Next, we set, for any $\alpha' \in (0,1-\alpha)$,
\[\phi_{q_+}(-s)=- \alpha'\frac{\Gamma(\alpha'(s+1)+1)}{ \Gamma(\alpha' s+1)}\]   
and we note from \cite[Section 3(1)]{Patie-aff} that $\psi^{q_+}(-s)=-s\phi_{q_+}(-s)= \frac{\Gamma(\alpha'(s+1)+1)}{ \Gamma(\alpha' s)}$  
 is  the Laplace exponent of a proper spectrally positive L\'evy process, and, writing $l=1/\alpha'$, we have 
\[ m_{\psi^{q_+}}(x) = l x^{-l-1}e^{-x^{-l}}.\]
On the one hand, from \eqref{eq:pse1}, we deduce  that  
\begin{eqnarray*}
m_{{\rm{I}}_{\tilde{\Psi}_q}}(x)
&=&\Gamma^{\alpha}(\gamma+1)\sum_{n=0}^{\infty}\frac{\Gamma(\alpha'(n+1)+1)}{ \Gamma^{\alpha}(-n+\gamma)}\frac{(-\alpha'x)^n}{n!} 
\end{eqnarray*}
which is easily seen to be absolutely convergent for all $x>0$ since $\alpha' \in (0,1-\alpha)$. This expression provides an expansion of the density for small values of the argument. In particular, we get that $   m_{{\rm{I}}_{\tilde{\Psi}_q}}(0) =   \gamma^{\alpha}\Gamma(\alpha'+1)  $. On the other hand, using the identity \eqref{eq:mpg}, we may also write, for any $x>0$,  
\begin{eqnarray*}
m_{{\rm{I}}_{\tilde{\Psi}_q}}(x)
&=&l  x^{-l -1}\int_0^{\infty}y^{l }e^{-(y/x)^{l 
}}m_{\phi_{\gamma}}(y)dy\\
&=&l  x^{-l  -1}
\sum_{n=0}^{\infty}\frac{(-1)^{n}x^{-l  n}}{n!}\int_0^{\infty}y^{l  (n+1)} m_{\phi_{\gamma}}(y)dy\\
&=&l  x^{-l  -1}
\sum_{n=0}^{\infty}\frac{\Gamma(l  (n+1)+1)\Gamma^{\alpha}(\gamma+1)}{\Gamma^{\alpha}(l  (n+1)+1+\gamma)}\frac{(-1)^{n}x^{-l  n}}{n!}\\
\end{eqnarray*}
to get an expansion of the density  for large values of its argument. Note, in particular,  that $\lim_{x\rightarrow \infty }x^{l  +1}m_{{\rm{I}}_{\tilde{\Psi}_q}}(x)
=l  \frac{\Gamma(l+1)\Gamma^{\alpha}(\gamma+1)}{\Gamma^{\alpha}(l+1+\gamma)}
$. 

\begin{remark}
The previous example illustrates nicely the fact that our main factorization  allows to get exact asymptotics for both large and small values of the argument as soon as one is able to expand as a series the density of the exponential functionals involved in the identity. 
\end{remark}

\noindent Finally, as a specific instance of  Corollary \ref{cor:main} (ii), we consider the case where $\Psi_q(s)=-\phi_\gamma(-s),s\geq0$, that is  $q=\gamma^{\alpha}$, $\delta=0$ and  $\Pi(dy)=\frac{\alpha}{\Gamma(1-\alpha)} e^{-\gamma y}y^{-\alpha-1}dy, y>0, \in \mathcal{P}$.  Thus, the series
\begin{eqnarray*}
 m_{\Psi_q}(x)&=&  \frac{\gamma^{\alpha}}{\Gamma(\gamma+1)}\sum_{n=0}^{\infty}\frac{\Gamma^{\alpha}(n+\gamma+1)}{n!}(-x)^n
\end{eqnarray*}
is absolutely convergent on $\R$ and completely monotone on $\R^+$.

\subsection{Proof of Corollary \ref{cor:stable}}
 According, for instance, to \cite[Proposition 4]{Kuznetsov-Pardo-11}, we have
\begin{equation} \label{eq:iti}
T_1 \stackrel{d}{=}  \int_0^{\e_q}e^{\xi_t} dt,
\end{equation}
where we have used the well known identity $T_1 \stackrel{d}{=} S^{-\alpha}_1$. Set $q=\frac{\Gamma(\alpha)}{\Gamma(\alpha\rho)\Gamma(1-\alpha\rho)}>0$ and note that $\xi$ is a L\'evy process with Laplace exponent $\Psi^{\alpha}$ given, for any $-1/\alpha<\Re(z)<1$, by
\begin{eqnarray*}
\Psi^{\alpha} (z)-q&=&-\frac{\Gamma(\alpha-\alpha z)\Gamma(\alpha z+1)}{\Gamma(\alpha\rho-\alpha z)\Gamma(\alpha z+1-\alpha\rho)}.
\end{eqnarray*}
First, let us consider the case when $\alpha \in (0, 1)$. We observe that $|\Psi^{\alpha} (s)|<+\infty$, for any $s\in [-1,0]$. Also we check that $ \Psi^{\alpha}(-1)-q\leq 0$ if $\frac{\Gamma(2\alpha)\Gamma(1-\alpha)}{\Gamma(\alpha(\rho+1))\Gamma(1-\alpha(\rho+1)}\geq0$ which is the case when $1-\alpha(\rho+1)\geq0$. Moreover, we know from \cite{Caballero} that for $0<\alpha<1$, the density of the L\'evy measure of $\xi$ restricted on $\R^+$ takes the form, up to a positive constant, $e^y(e^y-1)^{-\alpha-1}, y>0,$  which is easily seen to be decreasing on $\R^+$. Hence, we can apply Corollary \ref{cor:main} (i) to get that the density of $T_1$ is bounded and non-increasing. The boundedness property could have also been observed  when $\rho<1$  from \cite[Remark 5]{DoneySavov} which states that the density of $T_1$ has a finite limit at zero. Recalling that when $\rho=1$, we have $T_1\stackrel{d}{=}Z^{-\alpha}_1$, we could also  easily check from the Humbert-Pollard series representation of the density of $Z_1$, see e.g.~\cite[14.35]{Sato-99}, that the density of $T_1$ has also a finite limit at $0$. The remaining part of the  statement follows trivially.

\noindent Next, we assume that  $\alpha \in (1, 2]$ and $\rho= 1-\frac{1}{\alpha}$,  that is $Z$ is spectrally positive and thus $\xi$ is a spectrally negative L\'evy process with Laplace exponent, given, for any $\Re(z)>\frac{1}{\alpha}-1$, by
\begin{eqnarray*}
\Psi^{\alpha} (z)-q&=&\frac{\Gamma(\alpha z+1)}{\Gamma(\alpha z+1-\alpha)}.
\end{eqnarray*}
 Since $\Psi^{\alpha}\left(1-\frac{1}{\alpha}\right)-q=0$ we have $0<\gamma_q=1-\frac{1}{\alpha}\leq \frac{1}{2}$, we deduce from Corollary \ref{cor:main} (iv) that the density of $1/T_{1}$ is completely monotone which means that the density of  $S^{\alpha}_1$ is completely monotone. Note that the law of $S_1$ has been computed explicitly as a power series by Bernyk et al.~\cite{Bernyk-al}.
 We end up the paper by pointing out that in the Brownian motion case, i.e. $\alpha=2$,  the density of $S^{2}_1$ is well-known to be $m_{S^{2}_1}(x)= \frac{e^{-\frac{x}{2}}}{\sqrt{2\pi x}}$ and we get  $m_{S^{2}_1}(x)=\frac{1}{\sqrt{2}\pi}\int_{1/2}^{\infty}\frac{e^{-xr}}{\sqrt{r-1/2}}dr$.

\end{document}